\documentclass[11pt]{article}

\usepackage[a4paper,left=2.5cm,right=2.5cm,top=2.5cm,bottom=2.5cm]{geometry}
\usepackage{amsmath,amssymb,amsthm}
\usepackage{mathrsfs}
\usepackage{hyperref}

\newcommand{\Orb}{\operatorname{Orb}}
\newcommand{\NO}{\operatorname{NO}}

\newcommand{\B}{\mathcal{B}}

\newcommand{\N}{\mathbb{N}}

\newcommand{\Span}{\operatorname{span}}

\newtheorem{theorem}{Theorem}[section]
\newtheorem{proposition}[theorem]{Proposition}
\newtheorem{corollary}[theorem]{Corollary}
\newtheorem{lemma}[theorem]{Lemma}
\theoremstyle{definition}
\newtheorem{definition}[theorem]{Definition}
\newtheorem{example}[theorem]{Example}
\newtheorem{remark}[theorem]{Remark}
\newtheorem{question}[theorem]{Question}
\newtheorem{problem}[theorem]{Problem}

\title{On Topological Numerical Transitivity, C*-Transitivity and generalizations}

\author{%
Otmane Benchiheb\thanks{Department of Mathematics, Faculty of Sciences, Chouaib Doukkali University, El Jadida, Morocco. Emails: \texttt{benchiheb.o@ucd.ac.ma}, \texttt{otmane.benchiheb@gmail.com}.}
\and
Stefan Ivkovi\'c\thanks{Mathematical Institute of the Serbian Academy of Sciences and Arts, Kneza Mihaila 36, 11000 Belgrade, Serbia. Email: \texttt{stefan.iv10@outlook.com}.}
}
\date{}

\begin{document}

\maketitle

\begin{abstract}
Motivated by the concept of numerical hypercyclicity, in this paper we introduce three new notions in linear dynamics: topological numerical transitivity, generalized numerical transitivity and C*-transitivity. The first notion is defined for operators on general Banach spaces, whereas the latter two are formulated in the setting of operators on Banach algebras. The concept of C*-transitivity is also applicable for operators on the space of Hilbert-Schmidt operators. We prove that these new notions are mutually different and we also show that they differ  from the standard notions in dynamics like hypercyclicity, supercyclicity and numerical hypercyclicity. In particular, while topological numerical transitivity implies numerical hypercyclicity, as we argue in the paper, we provide examples of numerically hypercyclic and strongly numerically hypercyclic operators that are not topologically numerically transitive and examples of numerically hypercyclic operators that are neither C*-transitive nor generalized numerically transitive. Also, we construct non-supercyclic C*-transitive and generalized numerically transitive operators on the C*-algebra of compact operators on a separable Hilbert space, as well as C*-transitive operators on the standard Hilbert C*-module. In addition, we  study diagonal operators on finite-dimensional Hilbert spaces and on
\(\ell^p(\mathbb N)\), \(1<p<\infty\), and  we obtain complete characterizations of
both numerical hypercyclicity and topological numerical transitivity in terms
of coefficient-simplex criteria. Further, we provide some sufficient conditions for diagonal operators on the standard Hilbert C*-module to be C*-transitive. All these results are illustrated with concrete examples. At the end of the paper, we compare topological numerical transitivity and C*-transitivity with numerous standard notions in linear dynamics, like N-weak supercyclicity, weak-n supercyclicity, $\epsilon-$cyclicity, Li-Yorke chaos and distributional chaos of type 3, and we prove that they genuinely differ from all these standard concepts in dynamics. 
\end{abstract}



\section{Introduction}

Linear dynamics of bounded operators on Banach spaces has become a major theme
in modern functional analysis, at the crossroads of operator theory,
topological dynamics, and approximation theory
\cite{Birkhoff1929,MacLane1952,Rolewicz1969,Kitai1982,GodefroyShapiro1991,Herrero1991,Ansari1995,LeonsaavedraMuller2004,GrosseErdmann1999,BayartGrivaux2006,BesPeris1999,CostakisSambarino2004,BayartMatheron2009,GrosseErdmannPeris2011}.  
In the classical setting, an operator \(T\) acting on a topological vector space
\(X\) is called hypercyclic, if there exists a vector \(x\in X\) whose orbit under \(T\)
\[
\Orb(T,x):=\{T^n x:n\ge 0\}
\]
is dense in \(X\), whereas \(T\) is called supercyclic if there exists a vector \(x\in X\) whose  projective orbit under \(T\)
\[
\Orb_p(T,x):=\{\lambda T^n x:n\ge 0\ , \lambda \in \mathbb{C} \}
\]
is dense in \(X\). Further, we say that $T$ is {\it topologically transitive} on $X$ if for each pair of open non-empty subsets $O_{1}$ and $O_{2}$ of $X$ there exists some $n \in \mathbb{N}$ such that $  T^{n}(O_{1}) \cap O_{2} \neq \varnothing .$ Also, we say that $T$ is {\it topologically mixing} on $X$ if for each pair of open non-empty subsets $O_{1}$ and $O_{2}$ of $X$ there exists some $N \in \mathbb{N}$ such that $  T^{n}(O_{1}) \cap O_{2} \neq \varnothing $ for every $ n \geq N.$

Since the pioneering works of Birkhoff and Rolewicz, this
notion and its relatives have generated a rich theory, involving mixing,
weak mixing, recurrence, universality, frequent hypercyclicity, and spectral
criteria
\cite{Birkhoff1929,MacLane1952,Rolewicz1969,Kitai1982,GodefroyShapiro1991,Herrero1991,Ansari1995,LeonsaavedraMuller2004,BayartGrivaux2006,BonillaGrosseErdmann2007,CostakisSambarino2004,BayartMatheron2009,GrosseErdmannPeris2011}. Also, we refer to \cite{saw} for some recent results on topological transitivity.

A complementary viewpoint arises when one observes the dynamics not through
vector orbits in the ambient space, but through scalar evaluations of iterates.
This philosophy is classical in numerical range theory
\cite{Lumer1961,BonsallDuncan1971,BonsallDuncan1973}, and it led Kim, Peris,
and Song to introduce the notion of \emph{numerical hypercyclicity}
\cite{KimPerisSong}. Let \(X\) be a complex Banach space and define
\[
\Pi(X):=\{(x,x^*)\in S_X\times S_{X^*}:x^*(x)=1\}.
\]
For \(T\in\mathcal B(X)\) and \((x,x^*)\in\Pi(X)\), the associated numerical
orbit is
\[
\operatorname{NO}(T;x,x^*):=\{x^*(T^n x):n\ge 0\}\subseteq\mathbb C.
\]
The operator \(T\) is called numerically hypercyclic if there exists
\((x,x^*)\in\Pi(X)\) such that \(\operatorname{NO}(T;x,x^*)\) is dense in
\(\mathbb C\) \cite{KimPerisSong}. Thus the emphasis shifts from density of
vector orbits in \(X\) to density of scalar observations in the complex plane.

This scalar-dynamical viewpoint is markedly different from ordinary
hypercyclicity. In particular, Kim, Peris, and Song proved that every complex
finite-dimensional Banach space of dimension at least two supports numerically
hypercyclic operators, a phenomenon impossible in the classical hypercyclic
theory \cite{KimPerisSong}. In the same work, they obtained characterizations
of numerically hypercyclic weighted shifts on classical sequence spaces
\cite{KimPerisSong}. 

The theory was subsequently deepened by Shkarin in a substantial way
\cite{Shkarin}. In particular, he introduced the notions of \emph{weakly
numerically hypercyclic} and \emph{strongly numerically hypercyclic} operators,
thereby clarifying the role of similarity and showing that numerical
hypercyclicity is not, in general, similarity invariant
\cite{Shkarin}. His work also provided spectral criteria, finite-dimensional
descriptions, examples and counterexamples, and several structural results
concerning powers, scalar multiples, and concrete operator classes
\cite{Shkarin}. Altogether, these developments show that numerical
hypercyclicity is a subtle scalar version of orbit complexity, parallel to
classical hypercyclicity but with its own genuinely distinctive behavior
\cite{KimPerisSong,KimPerisSongPoly,Shkarin}.

The natural domain for this theory is the state space \(\Pi(X)\). For every
\(T\in\mathcal B(X)\) and every \(n\ge 0\), one has the numerical evaluation
map
\[
F_n:\Pi(X)\to\mathbb C,
\qquad
F_n(x,x^*)=x^*(T^n x).
\]
From this perspective, ordinary numerical hypercyclicity says precisely that
the sequence \((F_n)_{n\ge 0}\) has at least one universal point
\cite{KimPerisSong,Shkarin}. This observation places the subject naturally in
the broader framework of universality theory and Baire-category methods
\cite{GrosseErdmann1999,CostakisSambarino2004,BayartMatheron2009,GrosseErdmannPeris2011}.  
Indeed, in classical linear dynamics, topological transitivity is closely tied
to the existence of residual sets of hypercyclic vectors, and the same
philosophy suggests that one should go beyond the mere existence of a single
numerically hypercyclic pair and ask for a genuinely topological form of
numerical orbit complexity
\cite{GrosseErdmann1999,BayartMatheron2009,GrosseErdmannPeris2011}.

The purpose of the present paper is to initiate the systematic study of this
stronger phenomenon. Given \(T\in\mathcal B(X)\), we say that \(T\) is
\emph{topologically numerically transitive} if for every pair of nonempty open
sets
\[
U\subset\Pi(X),
\qquad
V\subset\mathbb C,
\]
there exists \(n\ge 0\) such that
\[
F_n(U)\cap V\neq\varnothing.
\]
Thus the question is no longer whether there exists one exceptional pair with
dense numerical orbit, but whether every nonempty open region of the state
space contains a pair whose numerical orbit visits any prescribed open subset
of the plane. In other words, we pass from an existence statement to a
genericity statement.

This notion is motivated by several converging ideas. First, it is formally
analogous to topological transitivity for operator iterates and to universality
for sequences of continuous maps
\cite{GrosseErdmann1999,CostakisSambarino2004,BayartMatheron2009,GrosseErdmannPeris2011}.  
Second, it fits naturally with the Baire-category structure of the state set
\(\Pi(X)\), where universal pairs may be organized into residual subsets rather
than treated merely as isolated objects. Third, it reflects a broader trend in
linear dynamics, where one seeks not only dense orbits, but generic orbit
behavior, frequency properties, and robust transitivity mechanisms
\cite{Ansari1995,BayartGrivaux2006,BonillaGrosseErdmann2007,BayartMatheron2009,GrosseErdmannPeris2011}.  
In this sense, topological numerical transitivity should be viewed as the
Baire-generic analogue of numerical hypercyclicity.

The main abstract result of the paper shows that this interpretation is exact:
under the natural topology on \(\Pi(X)\), topological numerical transitivity is
equivalent to the fact that the set of universal pairs is a dense
\(G_\delta\)-subset of \(\Pi(X)\). Hence the new notion is not merely a formal
variant of numerical hypercyclicity, but a genuinely stronger property, forcing
residual abundance of universal pairs rather than the existence of one such
pair. This also leads naturally to the converse problem: does numerical
hypercyclicity imply topological numerical transitivity? As our results show,
this question is already nontrivial in concrete classes such as weighted shifts
and diagonal operators. We provide in this paper examples of numerically hypercyclic diagonal operators that are \textit{not} topologically numerically transitive, giving thus negative answer to the above question. Also, we give some examples of topologically numerically transitive operators.

After developing the general Baire-category framework, we study diagonal operators in both finite-dimensional Hilbert
spaces and on \(\ell^p(\mathbb N)\) for \(1<p<\infty\). In these classes, the
problem reduces to coefficient-simplex criteria: numerical hypercyclicity
becomes an existence condition for a suitable scalar exponential sum, whereas
topological numerical transitivity becomes the corresponding density condition.
This dichotomy provides one of the main conceptual contributions of the paper
and shows that the new notion is both tractable and genuinely different from
ordinary numerical hypercyclicity.

Now, linear dynamics of operators on Banach algebras and their ideals
constitutes an active area of research; see, for instance,
\cite{chaotic-multipliers,ZZD,BIMS,AOFA,iran1,kina} and the references therein.
Motivated by this line of research and inspired by the concept of topological
numerical transitivity, we
introduce two new notions for bounded linear operators on Banach algebras and Banach C*-modules, which we call \emph{generalized numerical transitivity} and \emph{C*-transitivity}. The first notion is defined as follows:\\
Let $\mathcal{A}$ be a Banach algebra and let
$T\in\mathcal{B}(\mathcal{A})$. For each $n\in\mathbb{N}$, we
consider the induced map
\[
\theta_n:\mathcal{A}\times\mathcal{A}\longrightarrow\mathcal{A},
\qquad
\theta_n(a,b)=aT^n(b).
\]
We say that $T$ is generalized numerically transitive if the family
$\{\theta_n\}_{n\in\mathbb{N}}$ is topologically transitive when
$\mathcal{A}\times\mathcal{A}$ is equipped with the product topology.

For the second notion, we let $ \mathcal{M}$ be a Banach bimodule over a $C^*$-algebra $ \mathcal{C}$, and 
$ \mathcal{K}$ be a cone in $ \mathcal{M}$. We assume that $X$ is a Banach space such that
there exists a sesquilinear map
\[
\varphi:X\times X\longrightarrow \mathcal{M}
\]
with the property that
$
\varphi(x,x)\in \mathcal{K}
$
and such that
\[
\|x \|_X= \|\varphi(x,x) \|_{\mathcal{M}}^{1/2}
\qquad \text{for all }x\in X.
\] As we argue in the paper, C*-algebras, Hilbert C*-modules and the space of Hilbert-Schmidt operators on a separable Hilbert space are examples of such spaces.\\
Also, we let again $S(X)$ denote the unit sphere of $X,$ and for $T\in \mathcal{B}(X)$ and $n\in\N$ we 
let
\[
\widetilde{\Theta}_n:S(X)\times S(X)\longrightarrow \mathcal{M}
\]
be given by
\[
\widetilde{\Theta}_n(x,y)=\varphi\bigl(T^n x,y\bigr).
\]
Then we say that $T$ is \emph{topologically C*-
	transitive with respect to $\varphi $ } (or shortly \emph{C*-transitive with respect to $\varphi $}) if the sequence
$\{\widetilde{\Theta}_n\}_{n\in\N}$ is topologically transitive.

Our first result regarding these two new notions shows that, when $X$ is a
separable Banach space with the approximation property, then every
hypercyclic operator on the Banach algebra $K(X)$ of compact
operators on $X$ is generalized numerically transitive; see
Corollary~\ref{cor:compact-extension}. However, it turns out that the converse does not hold in general. In Proposition \ref{modular} we construct C*-transitive and generalized numerically transitive operators on $B_0(H)$ (the C*-algebra of compact operators on a separable Hilbert space) that are not even supercyclic and hence not hypercyclic.\\
Next, in Proposition \ref{implication} we show that in the special case of operators on the C*-algebra of compact operators on a Hilbert space, topological C*-transitivity implies generalized numerical transitivity, however, then in Proposition \ref{hilbert-module} we construct a C*-transitive operator on the standard Hilbert C*-module which is not generalized numerically transitive, illustrating in this way the difference between C*-transitivity and generalized numerical transitivity. In this context, we would like to emphasize that it has been proved in \cite{filomat} that the standard Hilbert C*-module is also a Banach algebra under componentwise multiplication, hence the concept of generalized numerical transitivity applies also in this concrete case. As a consequence of the proof of Proposition \ref{hilbert-module}, we provide some sufficient conditions for diagonal operators on the standard Hilbert C*-module to be C*-transitive (Corollary \ref{diag-Hilbert} ), and we illustrate this result with concrete examples (Remark \ref{diag-example}). Further, in Proposition \ref{prop:TNT-not-GNT-Schatten} we establish the existence of topological numerically transitive and numerically hypercyclic operators that are not C*-transitive. Finally, in the series of corollaries at the end of the paper, we show that there exist topologically numerically transitive operators that not weakly n-supercyclic nor N-weakly supercyclic for any $n, N \in \mathbb N,$ that are not $\epsilon-$cyclic for any $\epsilon \in (0,1),$ that are not Li-Yorke chaotic nor distributionally chaotic of type 3, and that are not subspace hypercyclic. Also, we show that there exist strongly numerically hypercyclic operators and subspace supercyclic operators that are not topologically numerically transitive and that are not C*-transitive.

These results show that C*-transitivity, topological numerical transitivity and generalized numerical transitivity are
genuinely different from the standard notions in
linear and numerical dynamics, and also, in general, they differ from each other. We hope that the introduction of these
concepts will motivate further investigations of their properties.

The paper is organized as follows. In Section~2, we introduce the state space
\(\Pi(X)\), fix the norm-product topology and some other notions needed for the rest of the paper. In Section~3.1, we define topological numerical
transitivity, establish a Baire-category criterion for it, and study its relation with ordinary numerical hypercyclicity.  In Section~3.2, we study diagonal
operators in finite-dimensional Hilbert spaces and on \(\ell^p(\mathbb N)\),
\(1<p<\infty\), obtaining complete characterizations in these settings.
In Section 4, we introduce generalized numerical transitivity and C*-transitivity for operators on Banach algebras and operator ideals. We compare these two new notions with each other and study their relations with many other concepts in linear dynamics. Also, we present examples showing that they are genuinely distinct from several known notions of numerical dynamics.

\section{Preliminaries}

Let \(X\) be a complex Banach space, let \(X^*\) be its dual, and denote by
\[
S_X:=\{x\in X:\|x\|=1\},
\qquad
S_{X^*}:=\{x^*\in X^*:\|x^*\|=1\}
\]
the unit spheres of \(X\) and \(X^*\). Following the standard framework of
numerical hypercyclicity, we consider the state space
\[
\Pi(X):=\{(x,x^*)\in S_X\times S_{X^*}:x^*(x)=1\}.
\]

For \(T\in\mathcal B(X)\) and \((x,x^*)\in\Pi(X)\), we write
\[
\NO(T;x,x^*):=\{x^*(T^n x):n\ge 0\}\subseteq\mathbb C
\]
for the associated numerical orbit. Recall that \(T\) is numerically
hypercyclic if there exists \((x,x^*)\in\Pi(X)\) such that \(\NO(T;x,x^*)\) is
dense in \(\mathbb C\).

Given \(T\in\mathcal B(X)\), we shall study the family of maps
\[
F_n:\Pi(X)\to\mathbb C,
\qquad
F_n(x,x^*)=x^*(T^n x),
\qquad n\in\mathbb N_0.
\]
We equip \(X\times X^*\) with the product norm
\[
\|(x,x^*)\|_\times:=\max\{\|x\|,\|x^*\|\},
\]
and endow \(\Pi(X)\) with the induced topology.

\begin{lemma}\label{prop:Pi-closed}
The set \(\Pi(X)\) is closed in \(S_X\times S_{X^*}\).
\end{lemma}

\begin{proof}
Consider
\[
\Phi:S_X\times S_{X^*}\to\mathbb C,
\qquad
\Phi(x,x^*)=x^*(x).
\]
It is straightforward to check that \(\Phi\) is continuous. Since
$
\Pi(X)=\Phi^{-1}(\{1\}),
$
the claim follows. \end{proof}

\begin{lemma}\label{prop:Pi-complete-baire}
The metric space \((\Pi(X),\|\cdot\|_\times)\) is complete. In particular,
\(\Pi(X)\) is a Baire space.
\end{lemma}

\begin{proof}
Since \(X\) and \(X^*\) are Banach spaces, the unit spheres \(S_X\) and
\(S_{X^*}\) are complete for the induced metrics. Hence
\(S_X\times S_{X^*}\) is complete for \(\|\cdot\|_\times\). By
Proposition~\ref{prop:Pi-closed}, \(\Pi(X)\) is a closed subset of
\(S_X\times S_{X^*}\), and therefore complete. The Baire property follows from
the Baire category theorem.
\end{proof}

  By some elementary computations it can be checked that for every \(n\in\mathbb N_0\) the map \(F_n\) is continuous and Lipschitz on \(\Pi(X)\)  with the Lipschitz constant $2\|T^n\|.$ The details in computations are left to the reader.

For later use, we will need the following definition.
\begin{definition}
For \(T\in\mathcal B(X)\), define
\[
\mathcal U(T):=
\Bigl\{(x,x^*)\in\Pi(X):\{F_n(x,x^*):n\ge 0\}\text{ is dense in }\mathbb C\Bigr\}.
\]
Equivalently,
\[
\mathcal U(T)=
\Bigl\{(x,x^*)\in\Pi(X):\NO(T;x,x^*)\text{ is dense in }\mathbb C\Bigr\}.
\]
\end{definition}

Thus \(\mathcal U(T)\) is precisely the set of universal pairs of \(T\), or
equivalently the set of numerically hypercyclic pairs.

Fix a countable basis \(\mathcal V=\{V_j:j\in\mathbb N\}\) of nonempty open
subsets of \(\mathbb C\). Then
\[
\mathcal U(T)=\bigcap_{j\in\mathbb N}\bigcup_{n\ge 0}F_n^{-1}(V_j).
\]
Indeed, \((x,x^*)\in\Pi(X)\) belongs to \(\mathcal U(T)\) if and only if its
numerical orbit meets every basis element \(V_j\).

\section{Topological numerical transitivity}
In the first subsection below we will study some general properties of topological numerical transitivity, whereas in the second subsection we will completely characterize topologically numerically transitive diagonal operators.

\subsection{A Baire-category criterion, and relations with numerical hypercyclicity }

Let \(T\in\mathcal B(X)\). We consider the family of maps
\[
F_n:\Pi(X)\to\mathbb C,
\qquad
F_n(x,x^*)=x^*(T^n x),
\qquad n\in\mathbb N_0.
\]

\begin{definition}\label{def:TNT}
An operator \(T\in\mathcal B(X)\) is called
\emph{topologically numerically transitive} if for every pair of nonempty open
sets
\(
U\subset\Pi(X)\),
\(V\subset\mathbb C,
\)
there exists \(n\in\mathbb N_0\) such that
\[
F_n(U)\cap V\neq\varnothing.
\]
Equivalently, for every such \(U\) and \(V\), there exist \((x,x^*)\in U\) and
\(n\in\mathbb N_0\) such that
\[
x^*(T^n x)\in V.
\]
\end{definition}

Fix a countable basis \(\mathcal V=\{V_j:j\in\mathbb N\}\) of nonempty open
subsets of \(\mathbb C\).

\begin{proposition}\label{prop:basis-formulation-TNT}
For \(T\in\mathcal B(X)\), the following assertions are equivalent:
\begin{enumerate}
\item[(i)] \(T\) is topologically numerically transitive;
\item[(ii)] for every nonempty open set \(U\subset\Pi(X)\) and every
\(j\in\mathbb N\), there exists \(n\in\mathbb N_0\) such that
\(
F_n(U)\cap V_j\neq\varnothing;
\)
\item[(iii)] for every \(j\in\mathbb N\), the set
\(
\bigcup_{n\ge 0}F_n^{-1}(V_j)
\)
is dense in \(\Pi(X)\).
\end{enumerate}
\end{proposition}

\begin{proof}
The implication (i)\(\Rightarrow\)(ii) is immediate.

Assume (ii). Let \(U\subset\Pi(X)\) and \(V\subset\mathbb C\) be nonempty open
sets. Choose \(j\in\mathbb N\) such that \(V_j\subseteq V\). Then
\(F_n(U)\cap V_j\neq\varnothing\) for some \(n\in\mathbb N_0\), hence
\(F_n(U)\cap V\neq\varnothing\). Thus (ii)\(\Rightarrow\)(i).

Fix \(j\in\mathbb N\). For a nonempty open set \(U\subset\Pi(X)\), the
condition
\(
F_n(U)\cap V_j\neq\varnothing
\)
for some \(n\in\mathbb N_0\) is equivalent to
\(
U\cap F_n^{-1}(V_j)\neq\varnothing
\)
for some \(n\in\mathbb N_0\), and hence to
\(
U\cap\bigcup_{n\ge 0}F_n^{-1}(V_j)\neq\varnothing.
\)
Therefore (ii) and (iii) are equivalent.
\end{proof}

Recall that
\(
\mathcal U(T)
=
\Bigl\{(x,x^*)\in\Pi(X):\NO(T;x,x^*)\text{ is dense in }\mathbb C\Bigr\},
\)
and, by the previous section,
\(
\mathcal U(T)
=
\bigcap_{j\in\mathbb N}\bigcup_{n\ge 0}F_n^{-1}(V_j).
\)

\begin{proposition}\label{thm:Baire-criterion}
For \(T\in\mathcal B(X)\), the following assertions are equivalent:
\begin{enumerate}
\item[(i)] \(T\) is topologically numerically transitive;
\item[(ii)] for every \(j\in\mathbb N\), the set
\(
\bigcup_{n\ge 0}F_n^{-1}(V_j)
\)
is open and dense in \(\Pi(X)\);
\item[(iii)] \(\mathcal U(T)\) is a dense \(G_\delta\) subset of \(\Pi(X)\).
\end{enumerate}
\end{proposition}

\begin{proof}
(i)\(\Rightarrow\)(ii). By Proposition~\ref{prop:basis-formulation-TNT}, the
sets
\(
\bigcup_{n\ge 0}F_n^{-1}(V_j)
\)
are dense in \(\Pi(X)\). They are open because each \(F_n\) is continuous.

(ii)\(\Rightarrow\)(iii). Since
\(
\mathcal U(T)
=
\bigcap_{j\in\mathbb N}\bigcup_{n\ge 0}F_n^{-1}(V_j),
\)
the set \(\mathcal U(T)\) is a countable intersection of open dense subsets of
the Baire space \(\Pi(X)\). Hence \(\mathcal U(T)\) is a dense \(G_\delta\)
subset of \(\Pi(X)\).

(iii)\(\Rightarrow\)(i). Let \(U\subset\Pi(X)\) and \(V\subset\mathbb C\) be
nonempty open sets. Choose \((x,x^*)\in U\cap\mathcal U(T)\). Since
\(\NO(T;x,x^*)\) is dense in \(\mathbb C\), there exists \(n\in\mathbb N_0\)
such that
\(
x^*(T^n x)\in V,
\)
that is, \(F_n(x,x^*)\in V\). Hence \(F_n(U)\cap V\neq\varnothing\).
\end{proof}

In particular,
\[
\text{topological numerical transitivity}
\Longrightarrow
\text{numerical hypercyclicity}.
\]

For later use, if \(V\subset\mathbb C\) is nonempty and open, define
\[
\mathcal O_T(V):=\bigcup_{n\ge 0}F_n^{-1}(V)
=
\Bigl\{(x,x^*)\in\Pi(X):\exists n\ge 0\text{ such that }x^*(T^n x)\in V\Bigr\}.
\]
Then Proposition~\ref{prop:basis-formulation-TNT} shows that \(T\) is
topologically numerically transitive if and only if \(\mathcal O_T(V_j)\) is
dense in \(\Pi(X)\) for every basis element \(V_j\), while
Proposition~\ref{thm:Baire-criterion} yields
\(
\mathcal U(T)=\bigcap_{j\in\mathbb N}\mathcal O_T(V_j).
\)

Thus topological numerical transitivity is exactly the condition ensuring that
the family \((F_n)\) has a residual set of universal points in \(\Pi(X)\).

We now compare topological numerical transitivity with the classical notion of
numerical hypercyclicity, and then record general obstruction principles
excluding large classes of operators.

The topological numerical transitivity is a property of the full family
\((F_n)_{n\ge 0}\), and not merely of a single numerical orbit. Thus the
distinction between the two notions is one of existence versus genericity.

\begin{proposition}\label{prop:parallel-formulations}
Let \(T\in\mathcal B(X)\). Then:
\begin{enumerate}
\item[(i)] \(T\) is numerically hypercyclic if and only if
\(
\mathcal U(T)\neq\varnothing;
\)
\item[(ii)] \(T\) is topologically numerically transitive if and only if
\(
\mathcal U(T)\text{ is residual in }\Pi(X);
\)
\item[(iii)] \(T\) is topologically numerically transitive if and only if
\(
\mathcal U(T)\text{ is dense in }\Pi(X).
\)
\end{enumerate}
\end{proposition}

\begin{proof}
Statement (i) is just the definition of numerical hypercyclicity. Statement
(ii) is exactly Theorem~\ref{thm:Baire-criterion}. Since \(\Pi(X)\) is a Baire
space, every residual subset is dense, so (ii) implies (iii). Conversely, if
\(\mathcal U(T)\) is dense in \(\Pi(X)\), then for every nonempty open sets
\(U\subset\Pi(X)\) and \(V\subset\mathbb C\), one can choose
\(
(x,x^*)\in U\cap\mathcal U(T),
\)
and the density of \(\NO(T;x,x^*)\) yields \(n\in\mathbb N_0\) such that
\(x^*(T^n x)\in V\). Hence \(F_n(U)\cap V\neq\varnothing\), so \(T\) is
topologically numerically transitive. Theorem~\ref{thm:Baire-criterion} then
gives the residuality of \(\mathcal U(T)\). Thus (ii) and (iii) are equivalent.
\end{proof}

Thus numerical hypercyclicity is an existence property, whereas topological
numerical transitivity is the corresponding density, equivalently residuality,
property for universal pairs.

It is also natural to compare the present notion with weak and strong numerical
hypercyclicity in the sense of Shkarin. These notions are defined through
similarity, whereas topological numerical transitivity is formulated through
the family \((F_n)\) on the fixed state space \(\Pi(X)\). Accordingly, they are
of different nature. Beyond the implication above, we do not claim here any
general relation between topological numerical transitivity and weak or strong
numerical hypercyclicity. In fact, in examples in the next section we will construct strongly numerically hypercyclic operator that is not topologically numerically transitive. We recall that an operator $T$ is said to be strongly numerically hypercyclic if every operator similar to $T$ is numerically hypercyclic. 

The converse implication is the first natural problem.

\begin{problem}\label{prob:NH-to-TNT}
Does numerical hypercyclicity imply topological numerical transitivity?
\end{problem}

Equivalently, if \(\mathcal U(T)\neq\varnothing\), must \(\mathcal U(T)\) be
dense in \(\Pi(X)\)? By Proposition~\ref{prop:parallel-formulations}, this is
exactly the question of whether one universal pair forces generic abundance of
universal pairs.
In the next section we will give negative answer to this question. In particular, we will construct concrete numerically hypercyclic and strongly numerically hhypercyclic operators that are not topologically numerically transitive.

For later use, we also record the following immediate positive criterion.

\begin{proposition}\label{prop:local-criterion}
Let \(T\in\mathcal B(X)\). If \(\mathcal U(T)\) is dense in \(\Pi(X)\), then
\(T\) is topologically numerically transitive.
\end{proposition}

\begin{proof}
This is Proposition~\ref{prop:parallel-formulations}(iii).
\end{proof}

\subsection{Diagonal operators: finite-dimensional and \texorpdfstring{$\ell^p$}{lp} cases}

Diagonal operators provide a natural class in which numerical dynamics reduce to
exponential sums with nonnegative coefficients. This yields complete
characterizations of both numerical hypercyclicity and topological numerical
transitivity in two settings: finite-dimensional Hilbert spaces and
\(\ell^p(\mathbb N)\) for \(1<p<\infty\).

We begin with the finite-dimensional Hilbert case. Let \(d\ge 2\), and let
\(
D=\operatorname{diag}(\lambda_1,\lambda_2,\dots,\lambda_d)\in\mathcal B(\mathbb C^d),
\)
where \(\lambda_1,\dots,\lambda_d\in\mathbb C\). Denote by
\(
\Delta_d:=\Bigl\{a=(a_1,\dots,a_d)\in[0,1]^d:\ \sum_{j=1}^d a_j=1\Bigr\}
\)
the standard simplex, and for \(a\in\Delta_d\) define
\(
\mathcal O_D(a):=\{1\}\cup\Bigl\{\sum_{j=1}^d a_j\lambda_j^n:\ n\ge 1\Bigr\}.
\)
Set
\(
G_D:=\Bigl\{a\in\Delta_d:\ \mathcal O_D(a)\text{ is dense in }\mathbb C\Bigr\}.
\)

\begin{theorem}\label{thm:finite-diagonal-characterization}
Let \(D=\operatorname{diag}(\lambda_1,\dots,\lambda_d)\in\mathcal B(\mathbb C^d)\).
Then:
\begin{enumerate}
\item[(i)] \(D\) is numerically hypercyclic if and only if \(G_D\neq\varnothing\).
\item[(ii)] \(D\) is topologically numerically transitive if and only if
\(G_D\) is dense in \(\Delta_d\).
\end{enumerate}
\end{theorem}

\begin{proof}
Since \(\mathbb C^d\) is a Hilbert space, \(\Pi(\mathbb C^d)\) is naturally
identified with the unit sphere \(S_{\mathbb C^d}\): for each
\(x\in S_{\mathbb C^d}\), the unique state functional associated with \(x\) is
\(
x^*(y)=\langle y,x\rangle.
\)
Thus, for \(x=(x_1,\dots,x_d)\in S_{\mathbb C^d}\),
\(
\langle D^0x,x\rangle=1,\)
 \(\langle D^n x,x\rangle=\sum_{j=1}^d |x_j|^2\lambda_j^n
\quad (n\ge 1).
\)
Hence the numerical orbit of \(x\) depends only on the coefficient vector
\(
a(x):=(|x_1|^2,\dots,|x_d|^2)\in\Delta_d,
\)
and one has
\[
x\in\mathcal U(D)\iff a(x)\in G_D.
\]
This proves (i).

Assume now that \(G_D\) is dense in \(\Delta_d\). Let
\(y=(y_1,\dots,y_d)\in S_{\mathbb C^d}\) and \(\varepsilon>0\). Set
\[
b:=a(y)=(|y_1|^2,\dots,|y_d|^2)\in\Delta_d,
\]
and choose \(a=(a_1,\dots,a_d)\in G_D\) such that
\(
\|a-b\|_1<\varepsilon^2.
\)
Define \(x=(x_1,\dots,x_d)\in\mathbb C^d\) by
\[
x_j=
\begin{cases}
\sqrt{a_j}\,\dfrac{y_j}{|y_j|},& y_j\neq 0,\\[2mm]
\sqrt{a_j},& y_j=0.
\end{cases}
\]
Then \(x\in S_{\mathbb C^d}\), \(a(x)=a\in G_D\), and hence
\(x\in\mathcal U(D)\). Moreover,
\[
\|x-y\|_2^2
=
\sum_{j=1}^d (|x_j|-|y_j|)^2
\le
\sum_{j=1}^d \bigl||x_j|^2-|y_j|^2\bigr|
=
\|a-b\|_1
<
\varepsilon^2.
\]
Thus \(\mathcal U(D)\) is dense in \(S_{\mathbb C^d}\), hence dense in
\(\Pi(\mathbb C^d)\). By Proposition~\ref{prop:parallel-formulations}(iii),
\(D\) is topologically numerically transitive.

Conversely, assume that \(D\) is topologically numerically transitive. Then
\(\mathcal U(D)\) is dense in \(S_{\mathbb C^d}\). Let
\(b=(b_1,\dots,b_d)\in\Delta_d\) and \(\varepsilon>0\), and set
\(
y=(\sqrt{b_1},\dots,\sqrt{b_d})\in S_{\mathbb C^d}.
\)
Choose \(x\in\mathcal U(D)\) such that
\[
\|x-y\|_2<\frac{\varepsilon}{2},
\]
and define
\(
a:=a(x)=(|x_1|^2,\dots,|x_d|^2)\in G_D.
\)
Then
\[
\|a-b\|_1
=
\sum_{j=1}^d \bigl||x_j|^2-|y_j|^2\bigr|
\le
\Bigl(\sum_{j=1}^d (|x_j|-|y_j|)^2\Bigr)^{1/2}
\Bigl(\sum_{j=1}^d (|x_j|+|y_j|)^2\Bigr)^{1/2}.
\]
Since \(\|x\|_2=\|y\|_2=1\), one has
\(
\sum_{j=1}^d (|x_j|+|y_j|)^2\le 4.
\)
Hence
\[
\|a-b\|_1
\le
2\,\||x|-|y|\|_2
\le
2\,\|x-y\|_2
<
\varepsilon.
\]
Thus \(G_D\) is dense in \(\Delta_d\), and the proof is complete.
\end{proof}

In the finite-dimensional diagonal Hilbert setting, numerical hypercyclicity is
therefore an existence property on \(\Delta_d\), whereas topological numerical
transitivity is the corresponding density property. We record three examples.

\begin{example}[A diagonal numerically hypercyclic operator]
Choose \(\lambda_1,\lambda_2\in\mathbb C\) such that
\(
\{\lambda_1^n+\lambda_2^n:n\ge 0\}
\)
is dense in \(\mathbb C\), and set
\(
D=\operatorname{diag}(\lambda_1,\lambda_2)\in\mathcal B(\mathbb C^2).
\)
Then
\(
a=\Bigl(\frac12,\frac12\Bigr)\in G_D,
\)
and hence \(D\) is numerically hypercyclic.
\end{example}

\begin{example}\label{TNT-3}(A diagonal topologically numerically transitive operator)
Let \(R>1\) and let \(u,w,z\in\mathbb T\) be independent. Consider
\(
D=\operatorname{diag}(Ru,Rw,Rz)\in\mathcal B(\mathbb C^3).
\)
By Shkarin's universality argument for three independent phases, the set
\[
\Bigl\{(a,b)\in\mathbb R^2:\ a>0,\ b>0,\ a+b<1,\
\{R^n(au^n+bw^n+(1-a-b)z^n):n\ge 0\}\text{ is dense in }\mathbb C\Bigr\}
\]
is a dense \(G_\delta\) subset of
\(
A=\{(a,b)\in\mathbb R^2:\ a>0,\ b>0,\ a+b<1\}.
\)
Therefore \(G_D\) is dense in \(\Delta_3\), and
Theorem~\ref{thm:finite-diagonal-characterization} implies that \(D\) is
topologically numerically transitive.
\end{example}

\begin{example}\label{NH-nonTNT}(A diagonal operator which is SNH but not TNT)
Fix \(1<R<M\), let \(u,w,z\in\mathbb T\) be independent, and define
\(
D=\operatorname{diag}(Ru,Rw,Rz,M)\in\mathcal B(\mathbb C^4).
\)
Choose \(a,b>0\) with \(a+b<1\) such that
\(
\{R^n(au^n+bw^n+(1-a-b)z^n):n\ge 0\}
\)
is dense in \(\mathbb C\). Then
\(
\alpha=(a,b,1-a-b,0)\in G_D,
\)
so \(D\) is numerically hypercyclic. In fact, by \cite[Theorem~1.10]{Shkarin} it is not hard to deduce that $D$ is even strongly numerically hypercyclic

If \(\beta=(\beta_1,\beta_2,\beta_3,\beta_4)\in\Delta_4\) satisfies
\(\beta_4>0\), then
\[
\sum_{j=1}^4\beta_j\lambda_j^n
=
R^n(\beta_1u^n+\beta_2w^n+\beta_3z^n)+\beta_4M^n.
\]
After division by \(M^n\),
\[
M^{-n}\sum_{j=1}^4\beta_j\lambda_j^n
=
\Bigl(\frac{R}{M}\Bigr)^n(\beta_1u^n+\beta_2w^n+\beta_3z^n)+\beta_4
\longrightarrow \beta_4.
\]
Hence the sequence cannot be dense in \(\mathbb C\). Thus every element of
\(G_D\) must satisfy \(\beta_4=0\), so
\(
G_D\subseteq\{\beta\in\Delta_4:\beta_4=0\},
\)
which is a proper closed subset of \(\Delta_4\). Therefore \(D\) is strongly numerically
hypercyclic but not topologically numerically transitive.
\end{example}

We now turn to diagonal operators on \(\ell^p(\mathbb N)\) in the reflexive
range \(1<p<\infty\). The same mechanism survives, with \(\Delta_d\) replaced by
the probability simplex of \(\ell^1\). Let
\[
D=\operatorname{diag}(\lambda_1,\lambda_2,\lambda_3,\dots)\in\mathcal B(\ell^p),
\qquad
\sup_{n\ge 1}|\lambda_n|<\infty.
\]
Denote by
\(
\Delta:=\Bigl\{a=(a_n)_{n\ge1}\in \ell^1:\ a_n\ge 0,\ \sum_{n\ge1}a_n=1\Bigr\}
\)
the probability simplex, and for \(a\in\Delta\) define
\(
\mathcal O_D(a):=\{1\}\cup\Bigl\{\sum_{n\ge1}a_n\lambda_n^m:\ m\ge 1\Bigr\}.
\)
Set
\(
G_D:=\Bigl\{a\in\Delta:\ \mathcal O_D(a)\text{ is dense in }\mathbb C\Bigr\}.
\)

\begin{theorem}\label{thm:general-diagonal-lp}
Let \(1<p<\infty\), let \(X=\ell^p(\mathbb N)\), and let
\(
D=\operatorname{diag}(\lambda_1,\lambda_2,\lambda_3,\dots)\in\mathcal B(X).
\)
Then:
\begin{enumerate}
\item[(i)] \(D\) is numerically hypercyclic if and only if \(G_D\neq\varnothing\).
\item[(ii)] \(D\) is topologically numerically transitive if and only if
\(G_D\) is dense in \(\Delta\).
\end{enumerate}
\end{theorem}

\begin{proof}
Let \(q\) be the conjugate exponent of \(p\). Since \(1<p<\infty\), every
\(x=(x_n)_{n\ge1}\in S_{\ell^p}\) has a unique norming functional
\(J_p(x)\in S_{\ell^q}\) given by
\[
J_p(x)_n=\overline{x_n}|x_n|^{p-2},
\qquad n\ge 1.
\]
Hence \(\Pi(X)\) is naturally identified with \(S_X\).

For \(x\in S_X\), define
\[
a(x):=(|x_n|^p)_{n\ge1}\in\Delta.
\]
Then
\[
J_p(x)(D^0x)=1,
\]
and for every \(m\ge 1\),
\[
J_p(x)(D^m x)
=
\sum_{n\ge1}\overline{x_n}|x_n|^{p-2}\lambda_n^m x_n
=
\sum_{n\ge1}|x_n|^p\lambda_n^m
=
\sum_{n\ge1}a(x)_n\lambda_n^m.
\]
Thus the numerical orbit of \(x\) is exactly \(\mathcal O_D(a(x))\), and
\[
x\in\mathcal U(D)\iff a(x)\in G_D.
\]
This proves (i).

Assume first that \(G_D\) is dense in \(\Delta\). Let \(y\in S_X\) and
\(\varepsilon>0\), and set
\[
b:=a(y)\in\Delta.
\]
Choose \(a\in G_D\) such that
\[
\|a-b\|_1<\varepsilon^p,
\]
and define \(x\in X\) by
\[
x_n=
\begin{cases}
a_n^{1/p}\,\dfrac{y_n}{|y_n|},& y_n\neq 0,\\[2mm]
a_n^{1/p},& y_n=0.
\end{cases}
\]
Then \(x\in S_X\), \(a(x)=a\in G_D\), and hence \(x\in\mathcal U(D)\).
Moreover,
\[
\|x-y\|_p^p
=
\sum_{n\ge1}\bigl||x_n|-|y_n|\bigr|^p
\le
\sum_{n\ge1}\bigl||x_n|^p-|y_n|^p\bigr|
=
\|a-b\|_1
<
\varepsilon^p.
\]
Thus \(\mathcal U(D)\) is dense in \(S_X\), hence dense in \(\Pi(X)\), and
Proposition~\ref{prop:parallel-formulations}(iii) yields topological numerical
transitivity.

Conversely, assume that \(D\) is topologically numerically transitive. Then
\(\mathcal U(D)\) is dense in \(S_X\). Let \(b\in\Delta\) and \(\varepsilon>0\),
and define
\[
y=(b_n^{1/p})_{n\ge1}\in S_X.
\]
Choose \(x\in\mathcal U(D)\) such that
\[
\|x-y\|_p<\frac{\varepsilon}{2p}.
\]
Set \(a:=a(x)\in G_D\). For nonnegative numbers \(s,t\),
\[
|s^p-t^p|\le p(s^{p-1}+t^{p-1})|s-t|.
\]
Applying this coordinatewise and summing, we obtain
\[
\|a-b\|_1
\le
p\sum_{n\ge1}\bigl(|x_n|^{p-1}+|y_n|^{p-1}\bigr)\bigl||x_n|-|y_n|\bigr|.
\]
By Hölder's inequality and \(\|x\|_p=\|y\|_p=1\),
\[
\|a-b\|_1
\le
2p\,\||x|-|y|\|_p
\le
2p\,\|x-y\|_p
<
\varepsilon.
\]
Thus \(G_D\) is dense in \(\Delta\).
\end{proof}

Thus, in the diagonal \(\ell^p\)-setting, numerical hypercyclicity is again an
existence property on the coefficient simplex, whereas topological numerical
transitivity is the corresponding density property.

We record three examples.

\begin{example}[A diagonal numerically hypercyclic operator]
Choose \(\lambda_1,\lambda_2\in\mathbb C\) such that
\[
\{\lambda_1^m+\lambda_2^m:m\ge 0\}
\]
is dense in \(\mathbb C\), and set
\[
D=\operatorname{diag}(\lambda_1,\lambda_2,0,0,\dots)\in\mathcal B(\ell^p).
\]
Then
\[
a=\Bigl(\frac12,\frac12,0,0,\dots\Bigr)\in G_D,
\]
and \(D\) is numerically hypercyclic.
\end{example}

\begin{example}[A diagonal topologically numerically transitive operator]
Let \(R>1\) and let \(u,w,z\in\mathbb T\) be independent. Consider
\[
D=\operatorname{diag}(Ru,Rw,Rz,Rz,Rz,\dots)\in\mathcal B(\ell^p).
\]
By similar arguments as in Example \ref{TNT-3} it can be checked that \(G_D\) is dense in \(\Delta\), hence
Theorem~\ref{thm:general-diagonal-lp} implies that \(D\) is topologically
numerically transitive.
\end{example}

\begin{example}\label{NHvsTNT}[A diagonal operator which is NH but not TNT]
Fix \(1<R<M\), let \(u,w,z\in\mathbb T\) be independent, and define
\[
D=\operatorname{diag}(Ru,Rw,Rz,M,0,0,\dots)\in\mathcal B(\ell^p).
\]
By similar arguments as in Example \ref{NH-nonTNT}  it can be deduced that every element of
\(G_D\) must satisfy \(\beta_4=0\), so
\[
G_D\subseteq\{\beta\in\Delta:\beta_4=0\},
\]
which is a proper closed subset of \(\Delta\). Therefore \(D\) is numerically
hypercyclic but not topologically numerically transitive.
\end{example}

The finite- and infinite-dimensional diagonal cases are therefore governed by
the same principle: numerical hypercyclicity corresponds to the existence of a
good coefficient distribution, whereas topological numerical transitivity
corresponds to density of such distributions in the ambient simplex. The
endpoint cases \(p=1\) and \(p=\infty\) require different methods and are left
open here.
\section{Generalized numerical and topological C*-transitivity in Banach algebras}
	Let $\mathcal{A}$ be a Banach algebra and $T \in \mathcal{B}(\mathcal{A})$.
For each $n\in\N$, we let
\[
\Theta_n:\mathcal A\times\mathcal A\longrightarrow \mathcal A
\]
be given by
\[
\Theta_n(a,b):=a\,T^n(b).
\]
Since the multiplication is jointly continuous on $\mathcal A$ and $T^n$ is bounded, it follows that $\theta_n$ is continuous for each $n\in\N$ when $\mathcal A\times\mathcal A$ is equipped with the product topology.

We will say that $T$ is \emph{generalized numerically transitive} if, for every non-empty open subsets $\widetilde O\subseteq \mathcal A\times\mathcal A$ and $O\subseteq\mathcal A$, there exists some $n\in\N$ such that
\[
\Theta_n(\widetilde O)\cap O\neq\varnothing.
\]

In the sequel, $H$ will denote a separable Hilbert space with an orthonormal basis
$
\{e_j\}_{j\in \mathbb{Z}},
$
and for each $m\in\N$ we let $P_m$ be the orthogonal projection onto
$
Span\{e_{-m},\ldots,e_m\}.
$
Further, $(B_0 (H),\|\cdot\|)$ will denote the $C^*$-algebra of compact operators on $H$ equipped with the operator norm.
Moreover, for each $p$ with $1\leq p<\infty$, we let $B_p(H)$ denote the corresponding Schatten $p$-ideal, equipped with the Schatten $p$-norm $\Vert\cdot\Vert_p$. 

We wish to show that every hypercyclic operator on $B_0 (H)$ is
generalized numerically transitive. For that purpose, we first need
the following auxiliary lemma.

\begin{lemma}\label{lem:finite-dimensional-perturbation}
	Let $X$ be a Banach space and let $X_0$ be a finite-dimensional
	subspace of $X$. Suppose that $T\in\mathcal{B}(X,X_0)$. Then there
	exist a finite-dimensional subspace $X_1\subseteq X$, with
	$\dim X_1=\dim X_0$, and a bounded projection $P:X\to X_1$ such that,
	for every $\varepsilon>0$, there exist operators
	$\widetilde{T}_{\varepsilon}\in\mathcal{B}(X_1,X_0)$ and
	$\widetilde{S}_{\varepsilon}\in\mathcal{B}(X_0,X_1)$ satisfying
	\[
	\left\|T-\widetilde{T}_{\varepsilon}P\right\|<\varepsilon,
	\qquad
	\widetilde{T}_{\varepsilon}\widetilde{S}_{\varepsilon}=I_{X_0}.
	\]
\end{lemma}

\begin{proof}
	Since $X_0$ is finite-dimensional, $T$ is a finite-rank operator.
	Thus $\operatorname{Ran}T$ is finite-dimensional and hence closed.
	Moreover, $\ker T$ has finite codimension in $X$. Hence there exists
	a finite-dimensional subspace $M\subseteq X$ such that
	$X=M\oplus\ker T$. The restriction
	$T|_M:M\to\operatorname{Ran}T$ is therefore an isomorphism.
	
	Since $X_0$ is finite-dimensional, choose a subspace $N\subseteq X_0$
	such that $X_0=\operatorname{Ran}T\oplus N$, and set
	$n=\dim N$.
	
	If $n=0$, then $\operatorname{Ran}T=X_0$. Set $X_1=M$, and let
	$P:X\to X_1$ be the bounded projection associated with
	$X=X_1\oplus\ker T$. Since $T$ vanishes on $\ker T$, we have $T=TP$.
	For every $\varepsilon>0$, define
	$\widetilde{T}_{\varepsilon}=T|_{X_1}$ and
	$\widetilde{S}_{\varepsilon}
	=\widetilde{T}_{\varepsilon}^{-1}$. Then
	$\widetilde{T}_{\varepsilon}\widetilde{S}_{\varepsilon}=I_{X_0}$
	and
	$\|T-\widetilde{T}_{\varepsilon}P\|=0<\varepsilon$.
	
	Suppose now that $n>0$. Choose linearly independent vectors
	$x_1,\ldots,x_n\in\ker T$ and set
	$Y=\operatorname{span}\{x_1,\ldots,x_n\}$. Then
	$\dim Y=\dim N$, so there exists an isomorphism $V:Y\to N$.
	
	Since $Y$ is finite-dimensional, it is complemented in $\ker T$.
	Thus there exists a closed subspace $Z\subseteq\ker T$ such that
	$\ker T=Y\oplus Z$. Consequently,
	\[
	X=M\oplus Y\oplus Z.
	\]
	Set $X_1=M\oplus Y$. Then
	$\dim X_1=\dim\operatorname{Ran}T+\dim N=\dim X_0$.
	
	Let $P:X\to X_1$ be the bounded projection associated with
	$X=X_1\oplus Z$, and let $Q:X_1\to Y$ be the projection onto $Y$
	along $M$. Since $VQP\neq0$, for a given $\varepsilon>0$ define
	$\delta_{\varepsilon}=\varepsilon/(2\|VQP\|)$ and define
	$\widetilde{T}_{\varepsilon}:X_1\to X_0$ by
	\[
	\widetilde{T}_{\varepsilon}(m+y)
	=
	Tm+\delta_{\varepsilon}Vy,
	\qquad m\in M,\quad y\in Y.
	\]
	
	Since $X_1=M\oplus Y$, $X_0=\operatorname{Ran}T\oplus N$,
	$T|_M:M\to\operatorname{Ran}T$ is an isomorphism, and $V:Y\to N$
	is an isomorphism, it follows that
	$\widetilde{T}_{\varepsilon}:X_1\to X_0$ is an isomorphism.
	Therefore, define
	$\widetilde{S}_{\varepsilon}
	=\widetilde{T}_{\varepsilon}^{-1}$.
	Hence
	$\widetilde{T}_{\varepsilon}\widetilde{S}_{\varepsilon}=I_{X_0}$.
	
	Finally, if $x=m+y+z$, with $m\in M$, $y\in Y$, and $z\in Z$,
	then $y,z\in\ker T$, so $Tx=Tm$, while $Px=m+y$. Hence
	\[
	T-\widetilde{T}_{\varepsilon}P
	=
	-\delta_{\varepsilon}VQP.
	\]
	Consequently,
	\[
	\left\|T-\widetilde{T}_{\varepsilon}P\right\|
	=
	\delta_{\varepsilon}\|VQP\|
	=
	\frac{\varepsilon}{2}
	<
	\varepsilon.
	\]
	This completes the proof.
\end{proof}

We now obtain the corresponding result for compact operators on
Banach spaces.

\begin{corollary}\label{cor:compact-extension}
	Let $X$ be a separable Banach space with the approximation property,
	and let $T\in\mathcal{B}(K(X))$ where $K(X)$ denotes the ideal of compact operators on $X$.
	If $T$ is hypercyclic on $K(X)$,
	then $T$ is generalized numerically transitive on $K(X)$.
\end{corollary}

\begin{proof}
	Let $O_1,O_2,O_3$ be non-empty open subsets of $K(X)$. Since $X$
	has the approximation property, finite-rank operators are dense in
	$K(X)$. We may therefore choose non-zero finite-rank operators
	$D\in O_1$, $F\in O_2$, and $G\in O_3$. Choose $\delta>0$ such that
	$B(D,\delta)\subseteq O_1$, $B(F,\delta)\subseteq O_2$, and
	$B(G,\delta)\subseteq O_3$.
	
	Set
	\[
	E=\operatorname{Ran}D+\operatorname{Ran}F+\operatorname{Ran}G.
	\]
	Then $E$ is a finite-dimensional subspace of $X$.
	
	Regard $D$ as an operator from $X$ into $E$. Applying
	Lemma~\ref{lem:finite-dimensional-perturbation} with
	$\varepsilon=\delta/2$, we obtain a finite-dimensional subspace
	$X_1\subseteq X$, a bounded projection $P_1:X\to X_1$, and operators
	$\widetilde D\in\mathcal{B}(X_1,E)$ and
	$\widetilde S\in\mathcal{B}(E,X_1)$ such that
	\[
	\|D-\widetilde D P_1\|<\frac{\delta}{2},
	\qquad
	\widetilde D\widetilde S=I_E.
	\]
	
	Define $\widehat D=\widetilde D P_1$. Since $\widehat D$ has
	finite-dimensional range, $\widehat D\in K(X)$, and the preceding
	estimate shows that $\widehat D\in O_1$. Moreover, since
	$P_1\widetilde S=\widetilde S$, we have
	$\widehat D\widetilde S=I_E$.
	
	Now set $Y=\widetilde S G$. Since $G$ has finite rank, $Y$ also has
	finite rank, and hence $Y\in K(X)$. Moreover,
	\[
	\widehat D Y
	=
	\widehat D\widetilde S G
	=
	G.
	\]
	
	Since $T$ is hypercyclic on $K(X)$, it is topologically transitive.
	Therefore there exist $B\in K(X)$ and $n\in\mathbb{N}$ such that
	\[
	\|B-F\|<\frac{\delta}{2},
	\qquad
	\|T^n(B)-Y\|
	<
	\frac{\delta}{2\|\widehat D\|}.
	\]
	Hence $B\in O_2$. Finally,
	$$
	\|\widehat D T^n(B)-G\|
	=
	\|\widehat D(T^n(B)-Y)\| \leq
	\|\widehat D\|\,\|T^n(B)-Y\| <
	\frac{\delta}{2}
	<
	\delta.$$
	Thus $\widehat D T^n(B)\in O_3$. Therefore the family
	$\{\Theta_n:n\in\mathbb{N}\}$, where
	$\Theta_n(A,B)=AT^n(B)$ for all $A$, $B\in K(X)$, is topologically transitive. Hence $T$ is
	generalized numerically transitive on $K(X)$.
\end{proof}

In particular, taking $X=H$, we have $K(H)=B_0 (H)$. Therefore,
every hypercyclic operator on $B_0 (H)$ is generalized numerically
transitive.

We will now introduce another notion of transitivity that is applicable in spaces whose norm is induced by a Banach bimodule-valued inner product. As we will see in examples, such spaces are for instance the space of Hilbert Schmidt operators and Hilbert C*-modules. Let $ \mathcal{M}$ be a Banach bimodule over a $C^*$-algebra $ \mathcal{C}$, and let
$ \mathcal{K}$ be a cone in $ \mathcal{M}$. Suppose that $X$ is a Banach space such that
there exists a sesquilinear map
\[
\varphi:X\times X\longrightarrow \mathcal{M}
\]
with the property that
\[
\varphi(x,x)\in \mathcal{K} \qquad \text{for all }x\in X,
\]
and such that
\[
\|x \|_X= \|\varphi(x,x) \|_{\mathcal{M}}^{1/2}
\qquad \text{for all }x\in X.
\]
Let $S(X)$ denote the unit sphere of $X$. For $T\in \mathcal{B}(X)$ and $n\in\N$,
let
\[
\widetilde{\Theta}_n:S(X)\times S(X)\longrightarrow \mathcal{M}
\]
be given by
\[
\widetilde{\Theta}_n(x,y)=\varphi\bigl(T^n x,y\bigr).
\]
We say that $T$ is \emph{topologically C*-
	transitive with respect to $\varphi $ } (or shortly \emph{C*-transitive with respect to $\varphi $}) if the sequence
$\{\widetilde{\Theta}_n\}_{n\in\N}$ is topologically transitive. Usually, in concrete cases  we will omit mentioning \emph{with respect to $\varphi $} when it is clear which sesquilinear map $\varphi $ is under consideration. Also, we will say that $T$ is \emph{topologically C*-
	mixing with respect to $\varphi $ } (or shortly \emph{C*-mixing with respect to $\varphi $}) if the sequence
$\{\widetilde{\Theta}_n\}_{n\in\N}$ is topologically mixing. Usually, in concrete cases  we will omit mentioning \emph{with respect to $\varphi $} when it is clear which sesquilinear map $\varphi $ is under consideration.

Before proceeding further, let us first give some examples of such Banach
spaces.
\begin{example}
	Let
	$
	X= \mathcal{M}=B_0(H),
	$
	and let $ \mathcal{K}$ be the natural cone consisting of positive operators in
	$B_0(H)$. Define
	$
	\varphi:B_0(H)\times B_0(H)\longrightarrow B_0(H)
	$
	by
	$
	\varphi(A,B)=B^*A,
	\text{ for all } A,B\in B_0(H).
	$
\end{example}
\begin{example}
	Let
	$
	X=B_2(H), \mathcal{M}=B_1(H),
	$
	and let $ \mathcal{K}$ be the natural cone consisting of positive operators in
	$B_1(H)$. Define
	$
	\varphi:B_2(H)\times B_2(H)\longrightarrow B_1(H)
	$
	by
	$
	\varphi(A,B)=B^*A,
	\text { for all } A,B\in B_2(H).
	$
\end{example}
In some special cases, topological C*-transitivity implies generalized numerical transitivity, as the next result illustrates.

\begin{proposition}\label{implication}
	If $T$ is a topologically  C*-transitive operator on $B_0(H)$, then
	$T$ is generalized numerically transitive on $B_0(H).$
\end{proposition}

\begin{proof}
	Let $O_1,O_2,O_3$ be non-empty open subsets of $B_0(H)$. Then we can find
	some
	$
	F\in O_1\setminus\{0\},
	\text{ }
	D\in O_2\setminus\{0\},
	$
	and some $\delta>0$ such that the $\delta$-neighbourhood
	$B(D,\delta)$ of $D$ is contained in $O_2\setminus\{0\}$.
	Note that
	$
	\bigl(B(D,\delta)\bigr)^*=B(D^*,\delta),
	$
	because the involution is an isometry. Clearly,
	$
	\frac{1}{ \|F\|}O_1\cap S(B_0(H))\neq\varnothing
	$
	and
	$$
	\frac{1}{ \|D \|}
	\bigl(B(D,\delta)\bigr)^*
	\cap S(B_0(H))\neq\varnothing,
	$$
	because $F\in O_1$ and $ \|D \|= \|D^* \|$.
	Since
	$
	\frac{1}{\|F\|}O_1,
	\text{ }
	\frac{1}{ \|D \|}B(D^*,\delta),
	\text{ }
	\frac{1}{\|F \| \|D \|}O_3
	$
	are open, if $T$ is topologically C*-transitive operator on $B_0(H)$, then we can find some
	$n\in\N$ and some
	\[
	\widetilde F\in
	\frac{1}{ \|F \|}O_1\cap S(B_0(H)),
	\widetilde D\in
	\frac{1}{ \|D \|}B(D^*,\delta)\cap S(B_0(H))
	\]
	such that
	$
	\widetilde D^*T^n(\widetilde F)
	\in
	\frac{1}{ \|F \| \|D \|}O_3.
	$
	Hence
	$
	 \|D \|\,\widetilde D^*
	T^n\bigl( \|F \|\widetilde F\bigr)
	\in O_3.
	$
	Now we have
	\[
	 \|D \|\,\widetilde D^*
	\in
	 \|D \|
	\left(
	\frac{1}{ \|D \|}B(D^*,\delta)
	\right)^*
	=B(D,\delta)
	\subseteq O_2,
	\]
	and
	$
	\|F \|\widetilde F\in O_1.
	$
	Thus $T$ is generalized numerically transitive.
\end{proof}

While every hypercyclic operator on $B_0(H)$ is generalized numerically transitive, it turns out that the converse does not hold in general. In what follows, we will construct a generalized numerically transitive operator on $B_0(H)$ which is not even supercyclic and hence is not hypercyclic.

Let $ \mathcal S$ be the bilateral forward shift on $H$, defined by
$ \mathcal S(e_j)=e_{j+1}$ for every $j\in\mathbb{Z}$, and set
\[
U=2 \mathcal S.
\]
Then $U$ is an invertible normal operator on $H$. We let $L_U$ be the
corresponding left multiplier by $U$ on $B_0(H)$, that is,
\[
L_U(F)=UF,
\qquad F\in B_0(H).
\]
In sequel, we will say that $L_U $ is WOT-hypercyclic if there exists some $F\in B_0(H) $ whose orbit under $L_U $ is dense in the weak operator topology of $ B_0(H).$

\begin{proposition}\label{modular}
	The operator $L_U$ is a non-supercyclic topologically C*-transitive and generalized numerically transitive operator on $B_0(H)$. 
	Moreover, $L_U$ is not WOT-hypercyclic.
\end{proposition}

\begin{proof}
	Let $O_1,O_2,O_3$ be non-empty open subsets of $B_0(H)$ such that
	$$
	O_j\cap S\bigl(B_0(H)\bigr)\neq\varnothing,
	\text { for all } j\in\{1,2,3\}.
	$$ Choose
	$
	D\in O_1\cap S\bigl(B_0(H)\bigr),
	F\in O_2\cap S\bigl(B_0(H)\bigr),
	G\in O_3\cap S\bigl(B_0(H)\bigr),
	$
	and choose some $\delta\in(0,1)$ such that the $\delta$-neighbourhoods
	of $D,F,G$ are contained in $O_1,O_2,O_3$, respectively.
	Now we recall that the involution is an isometry. Therefore, by \cite[Proposition 2.2.1]{MT} and by passing to the adjoints, it is not hard to see that we can find some $m\in\N$ such that
	\[
	\|D-P_mD \|,\quad
	\|F-P_mF \|,\quad
	 \|G-P_mG \|
	<\frac{\delta}{16(1+ \|G \|)}.
	\]
	Then
	\[
	\|D-P_mDP_m\|
	\leq
	\|D-P_mD\|+\|P_m\|\,\|D-DP_m\|
	<
	\frac{\delta}{8}.
	\]
	Hence,
	\[
	\|P_mDP_m \|
	\geq  \| D \|-\frac{\delta}{8}
	=1-\frac{\delta}{8}
	>\frac78>0,
	\]
	and similarly
	$
	 \| P_mG \|,
	 \|P_mF \|>1-\frac{\delta}{16}>0.
	$
	Moreover,
	$$
		\left\|D-\frac{P_mDP_m}{ \|P_mDP_m \|}\right\| \leq \|D-P_mDP_m \|
		+\left\|P_mDP_m-\frac{P_mDP_m}{\|P_mDP_m \|}\right\|$$ $$ \leq \frac{\delta}{8}
		+ \|P_mDP_m \|
		\left|1-\frac{1}{P_mDP_m \|}\right| =\frac{\delta}{8}+\left| P_mDP_m \|-1\right|
		<\frac{\delta}{4}.
	$$
	By Lemma \ref{lem:finite-dimensional-perturbation}, we can choose an invertible operator
	$\widetilde D\in B(P_m(H))$ such that
	\[
	\left\|
	\frac{P_mDP_m}{ \|P_mDP_m \|}-\widetilde D
	\right\|<\frac{\delta}{16}.
	\]
	Consequently,
	$
	\left\|D-\frac{\widetilde DP_m}{ \|\widetilde DP_m \|}\right\|
	<\delta,
	$
	so that
	$
	\frac{\widetilde DP_m}{ \|\widetilde DP_m \|}
	\in O_1\cap S\bigl(B_0(H)\bigr).
	$
	Again, we can find some $n_0\in\N$ such that, for every $n\geq n_0$,
	\[
	P_mU^nP_m=0
	\]
	and
	\[
	\|U^{-n}P_m \|
	<
	\frac{\delta}
	{4 \|\widetilde D^{*-1} \| \|\widetilde D^* \|
		\|G \|(1+ \|G \|)}.
	\]
	For each $n\geq n_0$, set
	\[
	E_n:=
	\frac{P_mF}{ \|P_mF \|}
	+U^{-n}\widetilde D^{*-1} \|\widetilde D^* \|\,P_mG.
	\]
	Since $P_m\widetilde DP_m=\widetilde D$, we also have
	$P_m\widetilde D^*P_m=\widetilde D^*$, and therefore
	\[
	U^{-n}\widetilde D^{*-1}
	=U^{-n}P_m\widetilde D^{*-1}.
	\]
	Hence
	\begin{align*}
		\|E_n-F\|
		&\leq
		\left\|F-\frac{P_mF}{ \|P_mF }\right\|
		+
		\|U^{-n}\widetilde D^{*-1} \|\widetilde D^* \|\,P_mG \| \\
		&\leq
		\|F-P_mF \|
		+\left| \|P_mF \|-1\right|
		+ \|U^{-n}P_m \| \|\widetilde D^{*-1} \|
		\|\widetilde D^* \| \|G\| \\
		&<\frac{3\delta}{8(1+ \|G \|)}.
	\end{align*}
	By the last inequality,
	$$\left\|\frac{E_n}{ \|E_n \|}-F\right\| \leq  \|E_n-F \|+\left|1- \|E_n \| \right| <\frac{3\delta}{4(1+ \|G \|)}<\delta.$$
	Thus,
	$
	\frac{E_n}{ \|E_n \|}
	\in O_2\cap S\bigl(B_0(H)\bigr).
	$
	Next, since $P_mU^nP_m=0$ for all $n\geq n_0$, by some computations we obtain
	\begin{align*}
		\varphi\left(
		L_U^n\left(\frac{E_n}{ \|E_n \|}\right),
		\frac{\widetilde D}{ \|\widetilde D \|}
		\right)
		&=
		\frac{P_mG}{ \|E_n \|},
	\end{align*}
	where we have used $ \|\widetilde D \|= \| \widetilde D^* \|$.
	Moreover,
	$$
		\left\|G-\frac{P_mG}{ \|E_n \|}\right\|\leq
		\|G-P_mG\|
		+
		\frac{ \|P_mG \|}{ \|E_n \|}
		\left| \|E_n \|-1\right| <
		\frac{\delta}{16}
		+
		\frac{\frac{3\delta}{8}}
		{1-\frac{3\delta}{8}}. $$
	Here we have used
	\[
	 \|G \|\,\left| \|E_n \|-1\right|
	\leq
	\left| \|E_n \|- \|F \| \right| \|G \|
	<\frac{3\delta}{8},
	\]
	and
	\[
	\|E_n \|
	> \|F\|-\frac{3\delta}{8(1+ \|G \|)}
	>1-\frac{3\delta}{8}.
	\]
	Since $\delta\in(0,1)$, we get
	$
	1-\frac{3\delta}{8}>\frac58,
	$  and hence
	\[
	\left\|G-\frac{P_mG}{ \|E_n\|}\right\|
	<\frac{\delta}{16}+\frac{3\delta}{5}<\delta.
	\]
	Therefore,
	\[
	\widetilde\Theta_n\!\left(
	O_1\cap S\bigl(B_0(H)\bigr),
	O_2\cap S\bigl(B_0(H)\bigr)
	\right)\cap O_3\neq\varnothing
	\]
	for every $n\geq n_0$. This proves that $L_U$ is topologically C*-
	transitive on $B_0(H)$. From Proposition \ref{implication} it follows that $L_U$ is generalized numerically
	transitive on $B_0(H)$.

It remains to show that $L_U$ is not supercyclic. Suppose, to the
contrary, that $L_U$ is supercyclic, and let $F\in B_0(H)$ be a
supercyclic vector for $L_U$. Since $F\neq0$, there exists $x\in H$
such that $Fx\neq0$.

Consider the evaluation map
$
\Phi_x:B_0(H)\longrightarrow H,$ given by $
\Phi_x(A)=Ax.
$
This map is continuous and surjective. Indeed, for every $y\in H$,
the rank-one operator
\[
A_y(z)=\frac{\langle z,x\rangle}{\|x\|^2}\,y
\]
belongs to $B_0(H)$ and satisfies $A_yx=y$.

Since the projective orbit of $F$ under $L_U$ is dense in
$B_0(H)$, its image under $\Phi_x$ is dense in $H$. Moreover,
\[
\Phi_x\bigl(\lambda L_U^n(F)\bigr)
=
\lambda U^n(Fx)
\]
for every $\lambda\in\mathbb{C}$ and $n\in\N$. Thus $Fx$ would
be a supercyclic vector for $U$. This is impossible because $U=2S$
is normal, and normal operators on complex Hilbert spaces are not
supercyclic by \cite[Theorem 4.5]{sanders}. Therefore $L_U$ is not supercyclic. \\
 Finally, by similar arguments and an application of \cite[Corollary 3.4]{feldman}, it can be proved that $L_U$ can neither be WOT-hypercyclic on $B_0(H).$  We leave the details to readers.
\end{proof}

As we will see in the next proposition, it is not true in general that topological C*-transitivity implies generalized numerical transitivity.

Let $\tilde\ell_2(B_0(H))$ denote the standard Hilbert module
over $B_0(H)$ (see \cite[Example 1.3.5]{MT}). Thus, in this
special case, we let
\[
X=\tilde\ell_2(B_0(H)), \qquad \mathcal{M}=B_0(H),
\]
and let $\varphi$ be the standard $B_0(H)$-valued inner product
on $\tilde\ell_2(B_0(H))$ given by
\[
\varphi\big((x_1,x_2,\ldots),(y_1,y_2,\ldots)\big)
=\sum_{j=1}^{\infty} y_j^*x_j.
\]
In other words,
\[
\varphi(x,y)=\langle y,x\rangle,
\qquad x,y\in\tilde\ell_2(B_0(H)),
\]
where $\langle\cdot,\cdot\rangle$ is the inner product on
$\tilde\ell_2(B_0(H))$ defined in \cite[Example 1.3.5]{MT}.

For every $x\in\tilde\ell_2(B_0(H))$ and $j\in\mathbb N$, we let
$x_j$ denote its $j$-th component. It has been proved in
\cite[Section 4]{filomat} that $\tilde\ell_2(B_0(H))$ is a Banach algebra under
componentwise multiplication, i.e.,
\[
(xy)_j=x_jy_j,
\qquad
x,y\in\tilde\ell_2(B_0(H)),\quad j\in\mathbb N.
\]
We have the following proposition.
\begin{proposition}\label{hilbert-module}
	There exists a topologically C*-transitive operator on
	$\tilde\ell_2(B_0(H))$ which is not generalized numerically transitive.
\end{proposition}

\begin{proof}
	Let $U$ be the normal operator on $H$ as above in Proposition \ref{modular}, and
	$
	\widetilde D\in\mathcal B\big(\tilde\ell_2(B_0(H))\big)
	$
	be given by
	\begin{equation}
		\widetilde D(T_1,T_2,\ldots)
		=(UT_1,0,0,\ldots),
		\qquad
		(T_1,T_2,\ldots)\in\tilde\ell_2(B_0(H)).
		\tag{2}
	\end{equation}
	
	We first show that $\widetilde D$ is C*-
	transitive on $\tilde\ell_2(B_0(H))$.
	To this end, let $O_1,O_2$ be open subsets of $\tilde\ell_2(B_0(H))$ such that
	\[
	O_j\cap S\big(\tilde\ell_2(B_0(H))\big)\neq\varnothing,
	\qquad j\in\{1,2\},
	\]
	and let $\widetilde O$ be a non-empty open subset of $B_0(H)$.
	Choose
	\[
	x\in O_1\cap S\big(\tilde\ell_2(B_0(H))\big),\qquad
	y\in O_2\cap S\big(\tilde\ell_2(B_0(H))\big),
	\]
	and $F\in\widetilde O$.
	By taking $\delta>0$ sufficiently small, we may assume that the
	$\delta$-neighbourhoods of $x$, $y$, and $F$ are contained in
	$O_1$, $O_2$, and $\widetilde O$, respectively.
	We may moreover assume that
	$x=(x_1,x_2,\ldots),\qquad
	y=(y_1,y_2,\ldots)$
	satisfy
	$
	x_1\neq0, y_1\neq0.
	$
	Indeed, suppose that $x_1=0$. Choose some
	$T\in S(B_0(H))$ and set
	\[
	x'
	=\left(\frac{\delta}{2}T,x_2,x_3,\ldots\right),
	\qquad
	x''=\frac{x'}{\|x'\|}.
	\]
	Then $x',x''\in\tilde\ell_2(B_0(H))$ and
	\[
	\|x-x''\|
	\leq \|x-x'\|+\|x'-x''\| =\|x-x'\|+\big|\|x'\|-1\big| <\delta.
	\]
	Hence
	\[
	x''\in O_1\cap S\big(\tilde\ell_2(B_0(H))\big),
	\]
	and the first component of $x''$ is non-zero. A similar
	argument applies to $y$.
	Since
	\[
	\lim_{t\to0}
	\left(
	\frac{1}{(1-t)^2}-1
	\right)=
	\lim_{t\to0}
	\left(
	1-\frac{1}{(1+t)^2}
	\right)=0,
	\]
	we can choose
	$\varepsilon\in\left(0,\frac{\delta}{4}\right)$
	such that
	\[
	1-\frac{1}{(1+\varepsilon)^2}, \text{ }
	\frac{1}{(1-\varepsilon)^2}-1
	<
	\frac{\delta}
	{4\left(
		\dfrac{\|F\|}{\|x_1\|\|y_1\|}+1
		\right)}.
	\]
	
	Let $V_1$ and $V_2$ be the $\varepsilon$-neighbourhoods of
	$
	\frac{x_1}{\|x_1\|}
	\text{ and }
	\frac{y_1}{\|y_1\|},
	$
	respectively, in $B_0(H)$. Also, let $\widetilde V$ be the
	$\varepsilon$-neighbourhood of
	$
	\frac{F}{\|x_1\|\|y_1\|}
	$
	in $B_0(H)$.
	Since $U$ is C*-transitive on $B_0(H)$ by Proposition \ref{modular},
	there exist some $n\in\mathbb N$ and
	\[
	z\in V_1\cap S(B_0(H)),\qquad
	w\in V_2\cap S(B_0(H))
	\]
	such that
	$
	w^*U^nz\in\widetilde V.
	$
	Thus,
	\[
	\left\|
	w^*U^nz-\frac{F}{\|x_1\|\|y_1\|}
	\right\|<\varepsilon.
	\]
	
	Since $\|x_1\|\leq\|x\|=1$, we have
	\[
	\left\|
	\|x_1\|z-x_1
	\right\|
	=
	\|x_1\|
	\left\|
	z-\frac{x_1}{\|x_1\|}
	\right\|
	<
	\|x_1\|\varepsilon
	<
	\varepsilon.
	\]
	Similarly,
	$
	\left\|
	\|y_1\|w-y_1
	\right\|<\varepsilon.
	$
	Let $\bar x,\bar y\in\tilde\ell_2(B_0(H))$ be given by
	\[
	\bar x
	=(\|x_1\|z,x_2,x_3,\ldots), \text{ }
	\bar y
	=(\|y_1\|w,y_2,y_3,\ldots).
	\]
	Then
	$
	\|x-\bar x\|<\varepsilon,
	\text{ }
	\|y-\bar y\|<\varepsilon.
	$
	Since $z,w\in S(B_0(H))$, both $\bar x$ and $\bar y$ are
	non-zero. Put
	$
	\widetilde x=\frac{\bar x}{\|\bar x\|}$ and $
	\widetilde y=\frac{\bar y}{\|\bar y\|}.$
	Then
	\[
	\|x-\widetilde x\|
	\leq
	\|x-\bar x\|
	+
	\left\|
	\bar x-\frac{\bar x}{\|\bar x\|}
	\right\| =
	\|x-\bar x\|
	+
	\big|\|\bar x\|-1\big| <2\varepsilon
	<\frac{\delta}{2}.
	\]
	Similarly,
	$
	\|y-\widetilde y\|<2\varepsilon<\frac{\delta}{2}.
	$
	Consequently,
	\[
	\widetilde x
	\in O_1\cap S\big(\tilde\ell_2(B_0(H))\big) \text{ and }
	\widetilde y
	\in O_2\cap S\big(\tilde\ell_2(B_0(H))\big).
	\]
	Furthermore,
	\[
	\left\langle
	\widetilde y,\widetilde D^n(\widetilde x)
	\right\rangle
	=
	\frac{\|y_1\|w^*}{\|\bar y\|}
	\,U^n\,
	\frac{\|x_1\|z}{\|\bar x\|}.
	\]
	Hence
	\[
	\left\|
	\left\langle
	\widetilde y,\widetilde D^n(\widetilde x)
	\right\rangle
	-F
	\right\|=
	\|x_1\|\|y_1\|
	\left\|
	\frac{w^*U^nz}{\|\bar y\|\|\bar x\|}
	-
	\frac{F}{\|x_1\|\|y_1\|}
	\right\|.
	\]
	Since $\|x_1\|,\|y_1\|\leq1$, it follows that
	\[ \left\|
	\left\langle
	\widetilde y,\widetilde D^n(\widetilde x)
	\right\rangle
	-F
	\right\| \leq
	\left\|
	\frac{w^*U^nz}{\|\bar y\|\|\bar x\|}-
	\frac{F}{\|x_1\|\|y_1\|}
	\right\| \]
	\[	\leq
	\left\|
	\frac{w^*U^nz}{\|\bar y\|\|\bar x\|}
	-w^*U^nz
	\right\|
	+
	\left\|
	w^*U^nz
	-
	\frac{F}{\|x_1\|\|y_1\|}
	\right\| \leq
	\|w^*U^nz\|
	\left|
	\frac{1}{\|\bar y\|\|\bar x\|}-1
	\right|
	+\frac{\delta}{4}.\]
	Moreover,
	\[
	\left\|
	w^*U^nz
	-
	\frac{F}{\|x_1\|\|y_1\|}
	\right\|
	<\varepsilon
	<\frac{\delta}{4}
	<1,
	\]
	and therefore
	\[
	\|w^*U^nz\|
	\leq
	\frac{\|F\|}{\|x_1\|\|y_1\|}+1.
	\]
	Thus
	\[ \left\|
	\left\langle
	\widetilde y,\widetilde D^n(\widetilde x)
	\right\rangle
	-F
	\right\|
	\leq
	\frac{\delta}{4}
	+
	\left(
	\frac{\|F\|}{\|x_1\|\|y_1\|}+1
	\right)
	\left|
	\frac{1}{\|\bar y\|\|\bar x\|}-1
	\right|. \]
	Now, since
	$
	\|x-\bar x\|<\varepsilon
	\text{ and }
	\|y-\bar y\|<\varepsilon,
	$
	we have
	\[
	1-\varepsilon
	=
	\|x\|-\varepsilon
	\leq
	\|\bar x\|
	\leq
	\|x\|+\varepsilon
	=
	1+\varepsilon,
	\]
	and similarly,
	$
	1-\varepsilon
	\leq
	\|\bar y\|
	\leq
	1+\varepsilon.$
	Therefore,
	$
	(1+\varepsilon)^{-2}
	\leq
	(\|\bar x\|\|\bar y\|)^{-1}
	\leq
	(1-\varepsilon)^{-2}.
	$\\
	If
	$
	(\|\bar x\|\|\bar y\|)^{-1}\leq1,$
	then
	$
	0
	\leq
	1-(\|\bar x\|\|\bar y\|)^{-1}
	\leq 1 -
	(1+\varepsilon)^{-2}.
	$
	On the other hand, if
	$
(\|\bar x\|\|\bar y\|)^{-1}\geq1,$
	then
	$
	0
	\leq
	(\|\bar x\|\|\bar y\|)^{-1}-1
	\leq
	(1-\varepsilon)^{-2}-1.
	$
	By our choice of $\varepsilon$, we therefore obtain
	\[
	\left|
	\frac{1}{\|\bar x\|\|\bar y\|}-1
	\right|
	<
	\frac{\delta}
	{4\left(
		\dfrac{\|F\|}{\|x_1\|\|y_1\|}+1
		\right)}.
	\]
	Hence,
	\[  \left\|
	\left\langle
	\widetilde y,\widetilde D^n(\widetilde x)
	\right\rangle
	-F
	\right\| \leq
	\frac{\delta}{4}
	+
	\left(
	\frac{\|F\|}{\|x_1\|\|y_1\|}+1
	\right)
	\left|
	\frac{1}{\|\bar y\|\|\bar x\|}-1
	\right| <
	\frac{\delta}{2}.\]
	Consequently,
	\[
	\left\langle
	\widetilde y,\widetilde D^n(\widetilde x)
	\right\rangle
	\in\widetilde O.
	\]
	Since $
	\widetilde x\in O_1, \text{ } \widetilde y\in O_2,$ and 
	$\|\widetilde x\|=\|\widetilde y\|=1,$
	we conclude that $\widetilde D$ is topologically C*-transitive
	on $\tilde\ell_2(B_0(H))$.
	
	Now we show that $\widetilde D$ is not generalized numerically
	transitive on $\tilde\ell_2(B_0(H))$.
	Let
	$ \mathcal R\in B_0(H)$ satisfy $
	\| \mathcal R\|=1$  and set
	$
	\eta=(0,\mathcal R,0,0,\ldots)\in\tilde\ell_2(B_0(H)).
	$
	Let $\widetilde{\mathcal U}$ be the $\frac12$-neighbourhood of
	$\eta$. For every $x\in\widetilde{\mathcal U}$, we have
	$
	\|x_2-\eta_2\|
	\leq
	\|x-\eta\|
	<
	\frac12.
	$
	Since $\eta_2= \mathcal R$ and $\| \mathcal R\|=1$, it follows that
	\[
	1-\|x_2\|
	\leq
	\| \mathcal R-x_2\|
	<
	\frac12,
	\]
	and hence $
	\|x_2\|>\frac12.$
	On the other hand, for every
	$y\in\tilde\ell_2(B_0(H))$, we have
	\[
	\widetilde D^{\,n}y
	=
	(U^ny_1,0,0,\ldots),
	\qquad n\in\mathbb N.
	\]
	Consequently,
	$
	\big(\widetilde D^{\,n}y\big)_2=0
	\text{ for every }n\in\mathbb N.
	$
	Since $\tilde\ell_2(B_0(H))$ is a Banach algebra under componentwise
	multiplication, we also have
$
	\big(z\widetilde D^{\,n}y\big)_2=0
	$
	for all $
	y,z\in\tilde\ell_2(B_0(H))$ and $ n\in\mathbb N.$
	Hence
	\[
	\Theta_n\big(
	\tilde\ell_2(B_0(H))\times\tilde\ell_2(B_0(H))
	\big)
	\cap
	\widetilde{\mathcal U}
	=
	\varnothing
	\qquad
	\text{for every }n\in\mathbb N.
	\]
	Consequently, $\widetilde D$ is not generalized numerically
	transitive on $\tilde\ell_2(B_0(H))$.
\end{proof}
The following corollary can be deduced by similar arguments as in the proof of Proposition \ref{hilbert-module}.
\begin{corollary}\label{diag-Hilbert}
	Let $T$ be a generalized compact (in the sense of \cite[Section 2.2]{MT}) operator on  $\tilde\ell_2(B_0(H))$ and suppose that $ T= diag(D_1, D_2, \cdots)$ with $ D_j \in B_0(H)$ for all $ j \in \mathbb N .$ Then, given any $ c \in (0,1) ,$ there exists some $ m \in \mathbb N $ such that $ \|  D_j \| \leq c $ for all $ j \geq m +1.$ Furthermore, if $ D_j $ is C*-mixing on $ B_0(H) $ for every $ j \leq m ,$ then $ T$ is C*-transitive on $\tilde\ell_2(B_0(H))$.
\end{corollary}
\begin{proof}
	The first statement follows directly from \cite[Proposition 2.2.1]{MT} because $T$ is generalized compact by the assumption. 
	Then, by utilizing this fact and by applying slight modifications of the proof of Proposition \ref{hilbert-module}, we can deduce the second statement of the corollary. We leave the details to readers.
\end{proof}
\begin{remark}\label{diag-example}
	We notice that by the proof of Proposition \ref{modular} it follows that $L_U$ is actually C*-mixing and not just C*-transitive. In general, the operator $ \lambda \mathcal S$ is C*-mixing on $B_0(H)$ for every scalar $ \lambda$ with $ \vert \lambda \vert > 1. $
\end{remark}
A natural question that arises now is whether there exist at all generalized numerically transitive operators on $\tilde\ell_2(B_0(H)).$ Positive answer to this question is provided in \cite[Proposition 2.9]{arxiv}.

The proof of Proposition \ref{hilbert-module} gives us also a tool for studying the relations between topological numerical transitivity, C*-transitivity and other dynamical concepts, which will be done in the next results. Although the concept of topological C*-transitivity is in a certain sense motivated by topological numerical transitivity, these two concepts are essentially different, as we show in the next proposition.
\begin{proposition}\label{prop:TNT-not-GNT-Schatten}
	There exists a topologically numerically transitive operator on
	$B_2(H)$ which is not topologically C*-transitive. In particular, there exists a numerically hypercyclic operator on $B_2(H)$ which is not topologically C*-transitive.
	
\end{proposition}

\begin{proof}
	Recall that $B_2(H)$ is a separable Hilbert space and hence is
	isometrically isomorphic to $\ell_2(\N)$. Let $D$ be the operator
	from Example~\ref{TNT-3}, which is topologically numerically transitive on $\mathbb C^3 ,$
	and $\{\psi_j\}_{j\in\N}$ be the standard orthonormal basis of
	$\ell_2(\N)\simeq B_2(H)$. Set
	\[
	V= Span\{\psi_1,\psi_2,\psi_3\}.
	\]
	If we let  $P$ denote now the orthogonal projection of $B_2(H)$ onto $V,$ then by similar calculations and arguments as in the proof of Proposition \ref{hilbert-module}, we can show that $D\circ P$ is topologically
	numerically transitive on $B_2(H)$. Hence $D\circ P$ is numerically hypercyclic on $B_2(H)$ since topological numerical transitivity implies numerical hypercyclicity.
	
	We now show that $D \circ P$ is not C*-transitive on $B_2 (H)$.
	Let $\{e_j\}_{j\in\N}$ be an orthonormal basis of $H$. Recall that
	$
	\{e_i\otimes e_j^*:i,j\in\N\}
	$
	is then an orthonormal basis of $B_2(H)$. Thus, we may assume that
	$\psi_k=e_{i_k}\otimes e_{j_k}^*$ for $k\in\{1,2,3\}$ and some
	$i_1,i_2,i_3,j_1,j_2,j_3\in\N$.
	Set
	\[
	H_0= Span\{e_{j_1},e_{j_2},e_{j_3}\}.
	\]
	Choose $x\in H_0^\perp$ with $\|x\|_H=1$, and let $ \tilde R$ be the
	orthogonal projection onto $ Span\{x\}$. Then $ \tilde R$ is a finite rank operator, hence it belongs to $B_1 (H)$.
	Let
	\[
	\widetilde U
	=
	\left\{
	T\in B_1(H):\|T- \tilde R\|_1<\frac12
	\right\}.
	\]
	Then $\widetilde U$ is a non-empty open subset of $B_1(H)$.
	For every $T\in\widetilde U$, we have
	\[
	1-\|T(x)\|_H
	\leq
	\| \tilde R(x)-T(x)\|_H
	\leq
	\| \tilde R-T\|
	\leq
	\|\tilde R-T\|_1
	<
	\frac12.
	\]
	Consequently,
	$
	\|T(x)\|_H>\frac12
	\text{ for every }T\in\widetilde U.
	$
	On the other hand, since $x\in H_0^\perp$, for each
	$k\in\{1,2,3\}$ we have
	\[
	(e_{i_k}\otimes e_{j_k}^*)(x)
	=
	\langle x,e_{j_k}\rangle e_{i_k}
	=
	0.
	\]
	Hence every operator in $V$ vanishes at $x$.
	Moreover, the construction of $D\circ P$ gives
	$\operatorname{Ran}(D\circ P)^n\subseteq V$ for every $n\in\N$.
	Therefore,
	$
	\bigl((D\circ P)^n(E)\bigr)(x)=0
	,$
	for every $E\in B_2(H)$ and every $n\in\N$,
	hence for each $n\in\N$ and every $F,E\in B_2(H)$, we have
	$
	\widetilde{\Theta}_n(F,E)(x)
	=
	F^*\bigl((D\circ P)^n(E)(x)\bigr)
	=
	0.
	$
	Thus $\widetilde{\Theta}_n(F,T)\notin\widetilde U$, because every element
	of $\widetilde U$ sends $x$ to a vector of norm greater than
	$1/2$. Consequently,
	\[
	\widetilde{\Theta}_n
	\bigl(B_2(H)\times B_2(H)\bigr)
	\cap\widetilde U
	=
	\varnothing
	\]
	for every $n\in\N$. Hence $D\circ P$ is not topologically C*-transitive on $B_2(H)$.
	This completes the proof.
\end{proof}
Recall that an operator
$
T\in \mathcal{B}(X),
$
where $X$ is a Banach space, is called \emph{weakly
	$n$-supercyclic} if there exists an $n$-dimensional subspace
$M$ of $X$ such that
\[
\bigcup_{k=1}^{\infty}T^k(M)
\]
is weakly dense in $X$.
Moreover, we recall from \cite{feldman}  that $T$ is said
to be \emph{$N$-weakly supercyclic} if there exists some
$x\in X$ such that, for every continuous surjective linear map
$
\gamma:X\longrightarrow \mathbb{C}^N,
$
it holds that
$
\gamma\bigl(\Orb_p(T,x)\bigr)
$
is dense in $\mathbb{C}^N$.

\begin{corollary}
	There exist $C^*$-transitive operators on
	$
	\widetilde{\ell}_2(B_0(H))
	$
	and topologically numerically transitive operators on
	$\ell^2(\N)$ which are neither weakly $n$-supercyclic nor
	$N$-weakly supercyclic, for any $n,N\in\N$.
\end{corollary}

\begin{proof}
	Let $\widetilde{D}$ be the operator on
	$\widetilde{\ell}_2(B_0(H))$ appearing in the proof of
	Proposition \ref{hilbert-module}.
	
	Let $M$ be any finite-dimensional subspace of
	$\widetilde{\ell}_2(B_0(H))$. Then
	\[
	\bigcup_{k=1}^{\infty}\widetilde{D}^{\,k}(M)
	\subseteq
	\widetilde{P}\bigl(\widetilde{\ell}_2(B_0(H))\bigr)
	\oplus
	(I-\widetilde{P})(M),
	\]
	where $\widetilde{P}$ denotes the projection onto the first
	coordinate of a sequence in $\widetilde{\ell}_2(B_0(H)).$
	Now, $(I-\widetilde{P})(M)$ is a finite-dimensional, hence
	closed subspace of
	$
	(I-\widetilde{P})
	\bigl(\widetilde{\ell}_2(B_0(H))\bigr).
	$
	Therefore, there exists a complemented subspace $\widetilde{M}$
	such that
	\[
	(I-\widetilde{P})
	\bigl(\widetilde{\ell}_2(B_0(H))\bigr)
	=
	(I-\widetilde{P})(M)\oplus\widetilde{M}.
	\]
	
	Clearly, $\widetilde{M}$ must be infinite-dimensional since
	$
	(I-\widetilde{P})
	\bigl(\widetilde{\ell}_2(B_0(H))\bigr)
	$
	is infinite-dimensional. By utilizing these facts and applying some elementary arguments from the functional analysis, it is not hard to show that $\widetilde{D}$ cannot be weakly n-supercyclic nor N-weakly supercyclic for any $n, N \in \mathbb{N} .$

	The proof is analogous for the operator $D\circ P$ appearing in
	the proof of Proposition \ref{prop:TNT-not-GNT-Schatten}. As established there,
	$D\circ P$ is a topologically numerically transitive operator
	on $B_2(H)$.
	By letting $D\circ P$, $P$, and $B_2(H)$ play, respectively,
	the roles of $\widetilde{D}$, $\widetilde{P}$, and
	$\widetilde{\ell}_2(B_0(H))$ in the arguments above, we deduce
	that $D\circ P$ cannot be weakly $n$-supercyclic or
	$N$-weakly supercyclic on $B_2(H)$ for any $n,N\in\N$.
	Since $B_2(H)$ is isometrically isomorphic to $\ell_2(\N)$,
	the second statement follows.
\end{proof}

Recall that for $\varepsilon\in(0,1)$, an operator $T\in \mathcal{B}(X)$ is said to be
$\varepsilon$-cyclic if there exists some $x\in X$ such that for every
$y\in X$ we can find some $ z \in \overline{
	\Span\left(
	\bigcup_{k=1}^{\infty}(T)^k(M)
	\right)}$ with
\[
\|z-y\|\leq \varepsilon\|y\|,
\] where $ M:= Span \lbrace x \rbrace .$

\begin{corollary}
	There exist topologically numerically transitive operators on $\ell^2(\mathbb N)$
	and $C^*$-transitive operators on $\widetilde{\ell}_2(B_0(H))$ that are not
	$\varepsilon$-cyclic for any $\varepsilon\in(0,1)$.
\end{corollary}

\begin{proof}
	Suppose that $x\in B_2(H)$ is an $\varepsilon$-cyclic vector for
	$D\circ P$, with $\varepsilon\in(0,1)$. Then, again,
	\[
	\overline{
		\Span\left(
		\bigcup_{k=1}^{\infty}(D \circ P)^k(M)
		\right)}
	\subseteq
	V\oplus (I-P)\bigl(\Span\{x\}\bigr).
	\]
	Let $h\in H$ with $\|h\|_H=1$ and choose some
	$
	g\in \Span\{e_{\lambda_1},e_{\lambda_2},e_{\lambda_3},x(h)\}^{\perp}
	$
	(we use the same notation as in the proof of Proposition~4.9) such that
	$\|g\|_H=1$, and let
	$
	y=g\otimes h^*,
	$
	so $y\in B_2(H)$. Then, clearly, for every $ z \in \overline{
		\Span\left(
		\bigcup_{k=1}^{\infty}(D \circ P)^k(M)
		\right)}$, we have
	\[
	1=\|g\|_H
	\leq
	\bigl\|\bigl(y-z\bigr)h\bigr\|_H,
	\]
	because
	$
	(D\circ P)^n(x)(h)
	\in \Span\{e_{\lambda_1},e_{\lambda_2},e_{\lambda_3},x(h)\}
	\text{ for all }n\in\mathbb N.
	$
	Consequently,
	$
	1=\|y\|\leq \|y- z\|
	$
	contradicting the fact that $x$ is an $\varepsilon$-cyclic vector for
	$D\circ P$.
	
	The approach is similar in the case of the operator $\widetilde D$ on
	$\widetilde{\ell}_2( B_0(H))$ from the proof of Proposition~4.8. Let $ \tilde x$ be In that case,
	\[
	\overline{
		\Span\left(
		\bigcup_{k=1}^{\infty}(\widetilde D)^k(\tilde M)
		\right)}
	\subseteq
	\widetilde P\bigl(\widetilde{\ell}_2(B_0(H))\bigr)
	\oplus
	(I-\widetilde P)\bigl(\Span\{\widetilde x\}\bigr),
	\]
	where we let now $\widetilde x$ be an $\varepsilon$-cyclic vector for
	$\widetilde D$,$ \tilde M := Span \lbrace \tilde x \rbrace$ and $\widetilde P$ is again the projection onto the first
	coordinate. Write
	$
	\widetilde x=(\widetilde x_1,\widetilde x_2,\widetilde x_3,\ldots)
	$
	and choose some $\widetilde g,\widetilde h\in H$ with
	\[
	\|\widetilde h\|_H=\|\widetilde g\|_H=1,
	\qquad
	\widetilde g\perp \widetilde x_2(\widetilde h).
	\]
	Let
	$
	\widetilde y=
	(0,\widetilde g\otimes\widetilde h^*,0,0,\ldots).
	$
	Then $\|\widetilde y\|=1$, and it can be checked that for every
	$n\in\mathbb N$,
	\[
	\|\widetilde D^n(\widetilde x)-\widetilde y\|
	\geq
	\bigl\|\bigl(\widetilde D^n(\widetilde x)-\widetilde y\bigr)_2
	(\widetilde h)\bigr\|_H
	\geq 1.
	\]
	Hence $\widetilde x$ cannot be an $\varepsilon$-cyclic vector for
	$\widetilde D$, a contradiction. Thus $\widetilde D$ cannot be
	$\varepsilon$-cyclic.
\end{proof}



We recall that an operator $T\in\B(X)$ is Li--Yorke chaotic if and only if
there exists some $x\in X$ and a strictly increasing sequence
$\{n_k\}\subseteq\mathbb N$ such that
\[
\lim_{k\to\infty}\|T^{n_k}(x)\|=0
\qquad\text{and}\qquad
\sup\{\|T^n(x)\|:n\in\mathbb N\}=\infty,
\]
whereas a necessary condition for $T$ to be distributionally chaotic of type~3
is that there exists some $y\in X$ and some $\delta>0$ such that both sets
$
A:=\{n\in\mathbb N:\|T^n x\|<\delta\}
$
and
$
B:=\{n\in\mathbb N:\|T^n x\|>\delta\}
$
are infinite, see the proof of part b) in \cite[Proposition 42]{distributionalchaos}.

\begin{corollary}
	There exist topologically numerically transitive operators on $\ell_2(\mathbb N)$,
	and $C^*$-transitive operators on $B_0(H)$, $ B_2(H)$ and
	$\widetilde{\ell}_2(B_0(H))$ that are not Li--Yorke chaotic nor
	distributionally chaotic of type~3.
\end{corollary}

\begin{proof}
	Note that since $U=2S$, where $S$ is the forward bilateral shift, for each
	$F\in B_0(H)$ we have
	$
	\|U^nF\|=2^n\|F\|,
	$
	and similarly for each $G\in B_2(H)$, it holds that 
	$
	\|U^nG\|_2=2^n\|G\|_2,
	$ for all $n\in\mathbb N.$
	Consequently, the conditions for Li--Yorke chaos and for distributional chaos
	of type~3 given above cannot be satisfied for the operator $L_U$.
	
	Further, if we consider the operator $\widetilde D$ on
	$\widetilde{\ell}_2(B_0(H))$ from the proof of Proposition \ref{hilbert-module}, then for every
	$x=(x_1,x_2,\ldots)\in\widetilde{\ell}_2(B_0(H))$ and each
	$n\in\mathbb N$, we have
	\[
	\bigl\|\widetilde D^n((x_1,x_2,\ldots))\bigr\|
	=2^n\|x_1\|.
	\]
	Therefore, it is again easy to deduce that $\widetilde D$ cannot be Li--Yorke
	chaotic nor distributionally chaotic of type~3.
	
	Finally, if we consider the operator $D\circ P$ on $B_2(H)$ from the proof of
	Proposition~4.9, then for every $F\in B_2(H)$ and each $n\in\mathbb N$ it holds
	that
	\[
	\|(D\circ P)^nF\|_2=R^n\|P(F)\|_2.
	\]
	Since $R>1$, it is again clear that $D\circ P$ can neither be Li--Yorke chaotic
	nor distributionally chaotic of type~3.
\end{proof}

\begin{corollary}\label{SNH-extension}
	There exist strongly numerically hypercyclic operators on $B_2(H)$, $B_0(H)$ and
	$\widetilde{\ell}_2(B_0(H))$ that are not $C^*$-transitive. Furthermore, there exist strongly numerically hypercyclic operators on \(\ell^2(\mathbb N)\) which are not topologically numerically transitive.
\end{corollary}

\begin{proof}
	Let $T$ be a strongly numerically hypercyclic operator on
	$
	V=\Span\{\psi_1,\psi_2,\psi_3\}.
	$
	We claim that $T\circ P$ is a strongly numerically hypercyclic operator on $B_2(H)$. To see this, let
	$J$ be an isomorphism from $B_2(H)$ onto another separable Hilbert space
	$\widetilde H$. Since $T\circ P$ and
	$
	J(T\circ P)J^{-1}
	$
	share the same eigenvalues, if $T$ is a diagonal operator with eigenvalues
	$\{\lambda_1,\lambda_2,\lambda_3\}$ satisfying the conditions of
	\cite[Theorem~1.10]{Shkarin}, then
	$
	J(T\circ P)J^{-1}\varphi
	$
	is a strongly numerically hypercyclic operator on
	$
	\Span\{J\psi_1,J\psi_2,J\psi_3\}
	$
	by \cite[Theorem~1.10]{Shkarin}, where we let $\varphi$ be the orthogonal projection
	onto $\Span\{J\psi_1,J\psi_2,J\psi_3\}$. In particular,
	$
	\varphi J(T\circ P)^nJ^{-1}\varphi
	$
	is numerically hypercyclic. However, since $\widetilde H$ is a Hilbert space, it is not hard to
	deduce that $J(T\circ P)J^{-1}$ is numerically on $\widetilde H$.
	
	Similar considerations apply in the case when we consider $V$ as a subspace
	of $B_0(H)$. Then $V$ is complementable in $B_0(H)$ since it is
	finite-dimensional. Let $\widehat P$ denote the projection onto $V$ along its
	complement in $B_0(H)$. Then $T\circ\widehat P$ is an operator on $B_0(H)$.
	By recalling the Hahn--Banach theorem regarding extensions of linear
	functionals, we can then apply similar arguments as above to deduce that
	$T\circ\widehat P$ is a strongly numerically hypercyclic operator on $B_0(H)$.
	
	However, by analogous arguments as in the proof of Proposition \ref{prop:TNT-not-GNT-Schatten}, we can
	conclude that $T\circ\widehat P$ is not $C^*$-transitive neither on $B_2(H)$
	nor on $B_0(H)$.
	
	Next, let $\widehat T$ be the operator on
	$\widetilde{\ell}_2(B_0(H))$ given by
	\[
	\widehat T(x_1,x_2,\ldots)
	=\bigl((T\circ\widehat P)(x_1),0,0,\ldots\bigr).
	\]
	Then, again by the arguments above and an application of the Hahn--Banach
	extension theorem, we obtain that $\widehat T$ is a strongly numerically hypercyclic operator on
	$\widetilde{\ell}_2(B_0(H))$. On the other hand, since
	\[
	\langle z,\widehat T y\rangle
	=z_1^*\bigl((T\circ\widehat P)(y_1)\bigr)
	\qquad
	\text{for all }z,y\in\widetilde{\ell}_2(B_0(H)),
	\]
	it follows that if $\widehat T$ is $C^*$-transitive on
	$\widetilde{\ell}_2(B_0(H))$, then $T\circ\widehat P$ must be
	$C^*$-transitive on $B_0(H)$. Since this is not the case, as we argued above,
	we obtain that $\widehat T$ is not $C^*$-transitive on
	$\widetilde{\ell}_2(B_0(H))$.
	
	Finally, by the arguments from the proof of Corollary \ref{SNH-extension} and from Example \ref{NH-nonTNT} it follows that the operator from Example  \ref{NHvsTNT} is actually strongly numerically hypercyclic, but not topologically numerically transitive, which proves the last statement of the proposition.
\end{proof}

Recall that an operator $T\in \B(X)$ is called
\emph{subspace hypercyclic} (respectively,
\emph{subspace supercyclic}) if there exist a closed subspace
$M\subseteq X$, with
$
\dim M\geq 2,
$
and some $x\in X$ such that
$
\Orb(T,x)\cap M
$
(respectively,
$
\Orb_{p}(T,x)\cap M
$)
is dense in $M$.


\begin{proposition}
	There exist topologically numerically transitive operators that are not subspace supercyclic.
	Furthermore, there exist $C^{*}$-transitive operators on
	$
	\widetilde{\ell}_{2}\bigl(B_{0}(H)\bigr)
	$
	that are not subspace hypercyclic, and there exist subspace
	supercyclic operators on
	$
	\widetilde{\ell}_{2}\bigl(B_{0}(H)\bigr)
	$
	that are not $C^{*}$-transitive.
\end{proposition}

\begin{proof}
	Let $D$ be the topologically numerically transitive operator from Example \ref{TNT-3}.
	Then, by \cite{HerziMarzougui2024} it follows that $D$ cannot be subspace supercyclic since it is a matrix on $ \mathbb C ^{3} .$
	Furthermore, let $\widetilde{D}$ be the operator on
	$
	\widetilde{\ell}_{2}\bigl(B_{0}(H)\bigr)
	$
	from Proposition \ref{hilbert-module}. Then
	$
	\left\|
	\widetilde{D}^{\,n}(x_{1},x_{2},\ldots)
	\right\|
	=
	2^{n}\|x_{1}\|
	$
	for all $n\in\mathbb{N}$ and
	$
	(x_{1},x_{2},\ldots)
	\in
	\widetilde{\ell}_{2}\bigl(B_{0}(H)\bigr).
	$
	Consequently, the norms of the elements of the orbit are either
	increasing or they are all equal to zero. By \cite{MadoreMartinez2011} (see the comments on page 511 regarding hyponormal operators in \cite{MadoreMartinez2011}), we deduce that
	that $\widetilde{D}$ cannot be subspace hypercyclic.
	
	Choose now some $h\in H$ with
	$
	\|h\|_{H}=1,
	$
	and let $p$ be the orthogonal projection onto
	$
	\operatorname{span}\{h\}.
	$
	It is easily checked that the map
	\[
	T:
	\widetilde{\ell}_{2}\bigl(B_{0}(H)\bigr)
	\longrightarrow
	\widetilde{\ell}_{2}\bigl(pB_{0}(H)p\bigr)
	\]
	given by
	$
	T(x_{1},x_{2},\ldots)
	=
	(px_{1}p,px_{2}p,\ldots)
	$
	is a bounded surjective linear map, and
	$
	T\big|_{\widetilde{\ell}_{2}(pB_{0}(H)p)}
	$
	is the identity on
	$
	\widetilde{\ell}_{2}\bigl(pB_{0}(H)p\bigr).
	$
	Let now $B$ denote any supercyclic operator on
	$
	\ell^{2}(\mathbb{N}).
	$
	Since
	$
	pB_{0}(H)p \cong \mathbb{C},
	$
	it follows that
	$
	\widetilde{\ell}_{2}\bigl(pB_{0}(H)p\bigr)
	$
	is isometrically isomorphic to $\ell^{2}(\mathbb{N})$.
	Let
	$
	\widetilde{T}:=BT.
	$
	Then
	$
	\widetilde{T}^{\,n}=B^{n}T
	\text{ for all } n\in\mathbb{N}.
	$
	Since $B$ is supercyclic on $\ell^{2}(\mathbb{N})$, we deduce that
	$\widetilde{T}$ is subspace supercyclic with respect to
	$
	\widetilde{\ell}_{2}\bigl(pB_{0}(H)p\bigr).
	$ However, it is straightforward to see that $ \left\langle
	\widetilde y,\widetilde T^n(\widetilde x)
	\right\rangle \in B_0(H)p,$ hence $\left\langle
	\widetilde y,\widetilde T^n(\widetilde x)
	\right\rangle \bigl((\operatorname{span}\{h\})^{ \perp } \bigr) = \lbrace 0 \rbrace$ for every $\widetilde x, \widetilde y \in \widetilde{\ell}_{2}\bigl(B_{0}(H)\bigr)  $ and $ n \in \mathbb N .$ By similar arguments as in the proof of Proposition \ref{prop:TNT-not-GNT-Schatten}, we can deduce that $ \widetilde T $ cannot be C*-transitive on $  \widetilde{\ell}_{2}\bigl(B_{0}(H)\bigr) .$
\end{proof}


We end the paper with the following question.

\begin{question}
	Does there exist a $C^{*}$-transitive operator on
	$
	\widetilde{\ell}_{2}\bigl(B_{0}(H)\bigr)
	$
	that is not subspace supercyclic?
\end{question}

\end{document}